\documentclass[11pt]{amsart}
\usepackage{amsmath,amsthm,verbatim,amssymb,amsfonts,amscd, graphicx}
\usepackage{graphics}
\usepackage{mathtools}
\usepackage{cite} 

\usepackage{xifthen} 
\usepackage{commath} 
\usepackage{booktabs} 
\usepackage{xcolor} 
\usepackage{xfrac} 
\usepackage{array} 
\usepackage{makecell}
\usepackage{multicol} 
\usepackage{pgfplots}
\pgfplotsset{compat=1.12} 
\usepackage{tikz}
\usetikzlibrary{positioning}
\usetikzlibrary{calc} 
\usetikzlibrary{matrix} 
\usetikzlibrary{intersections, through}
\usepackage[inline]{enumitem}
\usepackage{mathabx} 
\usepackage{tikz-cd}
\usepackage{wasysym} 
\usepackage{float} 
\usepackage{afterpage}
\floatstyle{boxed}
\restylefloat{figure}

 \usepackage[]{enumitem}
 \setlist[1]{wide}
 \setlist[2]{leftmargin=15mm}
 \setlist[enumerate]{label=\rm{(\arabic*)}}
 \setlist[enumerate,2]{label=\rm({\it\roman*}), }
 \setlist[itemize]{label=\raisebox{0.25ex}{\tiny$\bullet$}}

\usepackage[backref, colorlinks, linktocpage, citecolor = blue, linkcolor = blue]{hyperref} 

\theoremstyle{plain}

\newtheorem{theorem}{Theorem}

\newtheorem{lemma}{Lemma}[section]
\newtheorem{corollary}[lemma]{Corollary}
\newtheorem{proposition}[lemma]{Proposition}

\theoremstyle{definition}
\newtheorem{definition}[lemma]{Definition}

\theoremstyle{remark}
\newtheorem{remark}[lemma]{Remark}
\newtheorem{observation}[lemma]{Observation}
\newtheorem{example}[lemma]{Example}

\newcommand{\PPP}{\mathbb{P}}

\newcommand{\ZZZ}{\mathbb{Z}}
\newcommand{\NNN}{\mathbb{N}}

\newcommand{\CCC}{\mathbb{C}}

\newcommand{\FFF}{\mathbb{F}}
\newcommand{\QQQ}{\mathbb{Q}}

\newcommand{\PPPP}{\mathcal{P}}
\newcommand{\QQQQ}{\mathcal{Q}}
\newcommand{\NNNN}{\mathcal{N}}
\newcommand{\kk}{\mathbf{k}}

\newcommand{\simto}{\stackrel{\sim}\to}

\newcommand*\from{\colon}
\newcommand{\dashto}{\dashrightarrow}

\newcommand{\overbar}[1]{\mkern 1.5mu\overline{\mkern-1.5mu#1\mkern-1.5mu}\mkern 1.5mu}

\DeclareMathOperator{\Aut}{Aut}

\DeclareMathOperator{\Gal}{Gal}

\DeclareMathOperator{\id}{id}
\DeclareMathOperator{\Sym}{Sym}
\DeclareMathOperator{\Bir}{Bir}
\DeclareMathOperator{\BirMori}{BirMori}
\DeclareMathOperator{\Bas}{Bas}
\DeclareMathOperator{\NE}{NE} 
\DeclareMathOperator{\PGL}{PGL}
\DeclareMathOperator{\Pic}{Pic}
\DeclareMathOperator{\Cr}{Cr} 
\DeclareMathOperator{\rk}{rk} 

\DeclareMathOperator{\CB}{CB} 
\DeclareMathOperator{\Marked}{M} 

\begin{document}
\subjclass[2010]{Primary 14E07 14G27 14E05 14J26}
\keywords{Cremona group; normal subgroups; relations; conic bundles; Sarkisov links; Galois action; non-closed fields}

\title{Relations in the Cremona group over perfect fields}
\author[Julia Schneider]{Julia Schneider}
\address{Julia Schneider, Universit\"{a}t Basel, Departement Mathematik und Informatik, Spiegelgasse $1$, CH-$4051$ Basel, Switzerland}
\email{julia.noemi.schneider@unibas.ch}

\begin{abstract}
  For perfect fields $\kk$ satisfying $[\overbar \kk:\kk]>2$, we construct new normal subgroups of the plane Cremona group and provide an elementary proof of its non-simplicity, following the melody of the recent proof by Blanc, Lamy and Zimmermann that the Cremona group of rank $n$ over (subfields of) the complex numbers is not simple for $n\geq3$.
\end{abstract}

\maketitle
\tableofcontents

\section{Introduction}\label{section:CremonaOverPerfectIntroduction}

The Cremona group $\Cr_n(\kk)=\Bir_\kk(\PPP^n)$ is the group of birational transformations of the projective $n$-space over a field $\kk$.
In dimension $n=2$ it has been proven \cite{CL13, SB13, lonjou16} that the Cremona group over any field is not simple, answering a long-open question.
For algebraically closed fields, the classical result by Noether and Castelnuovo \cite{Noether1870, Castelnuovo1901}, which states that the Cremona group is generated by $\PGL_3(\kk)$ and the standard quadratic involution, has as a consequence that the Cremona group is a perfect group, meaning that all group homomorphisms from the Cremona group to an abelian group are trivial.
For many perfect fields, however, Lamy and Zimmermann constructed a surjective group homomorphism from the plane Cremona group to a free product of $\ZZZ/2\ZZZ$ \cite[Theorem C]{LZ19}, implying non-perfectness and thus reproving non-simplicity of the Cremona group in these cases.
Hence, the structure of the plane Cremona group depends fundamentally on the field.
Observing similarities between the plane Cremona group over non-closed fields and the Cremona groups in high dimensions, Blanc, Lamy and Zimmermann have recently managed to construct a surjective group homomorphism from the high-dimensional Cremona group $\Cr_n(\kk)$ to a free product of direct sums of $\ZZZ/2\ZZZ$, where $n\geq 3$ and $\kk\subset\CCC$ is a subfield \cite{BLZ19}.
For the high-dimensional case, it turned out that it is more suitable not to use the high-dimensional analogue of \cite{LZ19} but to take a different construction.
The goal of this article is to adapt the strategy of \cite{BLZ19} back to dimension two over perfect fields and find new normal subgroups of $\Cr_2(\kk)$.
No knowledge of \cite{BLZ19} is required to read our paper but we will highlight the connections to their proof.
In fact, only classical algebraic geometry is used, and the well-established decomposition of birational maps into Sarkisov links (proven in \cite[Theorem 2.5]{Iskovskikh96} and \cite[Appendix]{Corti95}, see Theorem~\ref{theorem:Iskovskikh} below).
Therefore, the following result can be seen as an elementary proof of the non-simplicity of the Cremona group over perfect fields whose extension degree of the algebraic closure is larger than $2$ (and thus infinite by Artin-Schreier):

\begin{theorem}\label{theorem:CremonaNotSimple}
  For each perfect field $\kk$ such that $[\overbar\kk:\kk]>2$, there exists a group homomorphism
  \[
    \Bir_\kk(\PPP^2)\twoheadrightarrow \bigoplus_I \ZZZ/2\ZZZ
  \]
  where the indexing set $I$ is infinite and whose kernel contains $\Aut_\kk(\PPP^2)=\PGL_3(\kk)$ such that the restriction to the subgroup that is locally given by
  \[
    \left\{(x,y)\mapsto (xp(y),y)\mid p\in\kk(y)\setminus\{0\}\right\}
  \]
  is surjective.
  In particular, the Cremona group $\Bir_\kk(\PPP^2)$ is not perfect and thus not simple.
\end{theorem}

The result is thus a $2$-dimensional analogue of \cite[Theorem A]{BLZ19}.
Theorem~\ref{theorem:CremonaNotSimple} together with \cite{zimmermann18}
(which states that the abelianisation of the real Cremona group is isomorphic to an uncountable direct sum of $\ZZZ/2\ZZZ$ and whose proof seems to work over any field with $[\bar\kk:\kk]=2$ when replacing ``uncountable'' with ``cardinality of the field'')
shows that for all perfect fields that are not algebraically closed, there is an infinite abelian group $A$, all of whose elements have order $2$, and a surjective group homomorphism $\Bir_\kk(\PPP^2)\dashrightarrow A$ whose kernel contains $\PGL_3(\kk)$.

While Noether and Castelnuovo provided a nice set of generators of the Cremona group over algebraically closed fields,
over non-closed fields such a nice set of generators is not known (a set of generators can be found in \cite{IKT93}).
For a Mori fiber space $X$ (a simple fibration, see Definition~\ref{def:MoriFiberSpace} below), however, we do not look at the group $\Bir(X)$ but consider instead the groupoid $\BirMori(X)$, which consists of birational maps between Mori fiber spaces that are birational to $X$.
The advantage is that we have generators:
The groupoid $\BirMori(X)$ is generated by Sarkisov links of type I to IV (simple birational maps, see Definition~\ref{def:SarkisovLink}) \cite[Theorem 2.5]{Iskovskikh96}.
Details about generators and relations for groupoids can be found in  \cite{higgins64}.
Whereas over an algebraically closed field the Sarkisov links are just the blow-up of one point (type III), or its inverse (type I), or the blow-up of one point followed by the contraction of one curve (type II), or an exchange of the fibration of $\PPP^1\times\PPP^1$ (type IV), over a perfect field one has to consider orbits of the Galois action of $\Gal(\overbar\kk/\kk)$ on $X$.
In this paper, the size of the orbits that lie in the base locus of a birational map is going to be important.
We say that the \textit{Galois depth} of a birational map $\varphi$ is the maximal size of an orbit that lies in the base locus of $\varphi$ or $\varphi^{-1}$.

So we do not directly construct a group homomorphism from $\Bir(X)$, but we first construct a groupoid homomorphism from $\BirMori(X)$ and then take the restriction to $\Bir(X)$.
For this one has to study relations in the groupoid.
Note that in \cite{Iskovskikh96} there is a long and complicated list of generating relations.
In \cite{LZ19} the focus lies on Bertini involutions (the blow-up of an orbit of size $8$ in $\PPP^2$, followed by the contraction of an orbit of curves of size $8$).
For higher dimensions, the focus lies on links of type II between Mori conic bundles that have a large covering gonality (see \cite{BLZ19} for definitions).
Translating this back to the $2$-dimensional case, we focus on links of type II between Mori conic bundles (that is, a Mori fiber space $X\to B$ where $B$ is a curve) that have large Galois depth and find the following generating relations:

\begin{theorem}\label{theorem:GeneratingRelations}
  Let $X$ be a Mori fiber space (of dimension $2$) over a perfect field.
  Relations of the groupoid $\BirMori(X)$ (with Sarkisov links as a set of generators) are generated by the trivial relations and relations of the following form:
  \begin{enumerate}
    \item\label{item:GeneratingRelations--Small}  $\varphi_n\cdots\varphi_1=\id$, where the Galois depth of all $\varphi_i$ is $\leq 15$, and
    \item\label{item:GeneratingRelations--Large}  $\chi_4\chi_3\chi_2\chi_1=\id$ where  $\chi_i\colon X_{i-1}\dashto X_i$ are links of type II between Mori conic bundles such that the links $\chi_1$ and $\chi_3$ are equivalent
    and $\chi_2$ and $\chi_4$ are equivalent.
  \end{enumerate}
\end{theorem}

(What is meant by trivial relations is explained at the beginning of Section~\ref{section:Relations}.
For the notion of equivalent links see Remark~\ref{remark:RelationOfFourLinks} and Definition~\ref{def:EquivalentSarkisovLinks}.)
This can be compared with \cite[Proposition 5.5]{BLZ19}.
We give a simple and self-contained proof to the above theorem using elementary techniques.
One could also give a proof of Theorem~\ref{theorem:GeneratingRelations} with $8$ instead of $15$, using the work of \cite{LZ19} (or \cite{IKT93}) and the fact that the blow-up of more than $8$ points is not del Pezzo anymore, but we do not achieve this with the elementary approach that we follow here: Some technicalities in Lemma~\ref{lemma:InductionStepLambda} deny us this pleasure.
One may however observe that in dimension $n\geq 3$ the bound on the covering gonality given in \cite{BLZ19}, the analogue of the Galois depth, is not explicit.
Using these generating relations, we are finally able to construct a groupoid homomorphism. (For the notation: $\CB(X)$ denotes the set of equivalence classes of Mori conic bundles, and $\Marked(C)$ denotes the set of equivalence classes of Sarkisov links between Mori conic bundles equivalent to $C$; see Definitions~\ref{def:EquivalentConicBundles} and~\ref{def:EquivalentSarkisovLinks}.)

\begin{theorem}\label{theorem:GroupHomomorphism}
  Let $X$ be a Mori fiber space over a perfect field.
  There exists a groupoid homomorphism
  \[
  \BirMori(X)\longrightarrow\bigast_{C\in \CB(X)}\bigoplus_{\chi\in \Marked(C)}\ZZZ/2\ZZZ
  \]
  that sends each Sarkisov link $\chi$ of type II between Mori conic bundles that is of Galois depth $\geq 16$ onto the generator indexed by its equivalence class, and all other Sarkisov links and all automorphisms of Mori fiber spaces birational to $X$ onto zero.

  Moreover it restricts to group homomorphisms
  \begin{align*}
    \Bir(X)\to \bigast_{C\in \CB(X)}\bigoplus_{\chi\in \Marked(C)}\ZZZ/2\ZZZ,&&
    \Bir(X/W)\to \bigoplus_{\chi\in \Marked(X/W)}\ZZZ/2\ZZZ.
  \end{align*}
\end{theorem}

This is analogue to \cite[Theorem D]{BLZ19}.
Note that the group homomorphism of Theorem~\ref{theorem:GroupHomomorphism} is trivial if the field does not admit large orbits.
If the group homomorphism is not trivial, the kernel is a non-trivial normal subgroup of $\Bir(X)$.
For perfect fields $\kk$ that admit a large orbit and $X=\PPP^2_\kk$, we restrict the obtained group homomorphism to the equivalence class of Mori conic bundles corresponding to the Hirzebruch surfaces and provide an example that shows that the group homomorphism is not trivial and therefore obtain Theorem~\ref{theorem:CremonaNotSimple}.
Here ends the first part of this paper.

\bigskip
In the final section we investigate rational Mori conic bundles in order to refine Theorem~\ref{theorem:CremonaNotSimple}:

\begin{theorem}[Refinement of Theorem~\ref{theorem:CremonaNotSimple}]\label{theorem:Refinement}
  For each perfect field $\kk$ such that $[\overbar\kk:\kk]>2$, there exists a surjective group homomorphism
  \[
    \Bir_\kk(\PPP^2)\twoheadrightarrow \bigoplus_{I_0} \ZZZ/2\ZZZ \asterisk \left(\bigast_{J_5} \bigoplus_{I}\ZZZ/2\ZZZ\right) \asterisk \left(\bigast_{J_6} \bigoplus_{I}\ZZZ/2\ZZZ\right)
  \]
  where $I_0$ is infinite and $I$ is the set of integers $n\geq8$ such that there exists an irreducible polynomial in $\kk[x]$ of degree $2n+1$, and the indexing sets $J_d$ of the free products denote a set of equivalence classes of Mori conic bundles $X$ with $K_X^2=d$.
  The sets have cardinalities $|J_5|\geq\mathcal{N}_{4}$ and $|J_6|\geq\mathcal{N}_{2}$, where $\mathcal{N}_n$ denotes the number of Galois orbits of size $n$ in $\PPP^2$, up to linear transformations of $\PGL_3(\kk)$.

  In particular, if $\kk=\QQQ$ then there is a surjective group homomorphism  \[\Bir_\kk(\PPP^2)\twoheadrightarrow \bigast_{\NNN}\ZZZ/2\ZZZ,\]
  and for any finite field $\kk$ there is a surjective group homomorphism    \[\Bir_\kk(\PPP^2)\twoheadrightarrow \ZZZ/2\ZZZ\ast\ZZZ/2\ZZZ\ast\ZZZ/2\ZZZ.\]
\end{theorem}

The existence of a group homomorphism from $\Bir_\kk(\PPP^2)$ to a free product of at least three copies of $\ZZZ/2\ZZZ$ implies SQ-universality of the plane Cremona group over finite fields and the field of rational numbers (see \cite[Theorem 3]{schupp73} and \cite[Corollary 8.4]{BLZ19}).
This is new for $\kk=\FFF_2$ but for many fields it is already given in \cite{LZ19}, though with a different construction of the homomorphism.
(The kernel in Theorem~\ref{theorem:Refinement} contains all Bertini involutions, whereas the kernel in \cite{LZ19} contains all de Jonqui\`eres maps.)
In this final section, however, we do not only use \cite[Theorem 2.5]{Iskovskikh96} but also the above mentioned long list of Sarkisov links in the same paper.
\bigskip

The paper is structured as follows:
In Section~\ref{section:CremonaOverPerfectPreliminaries} we introduce the notion of Mori fiber space and Sarkisov link, and state some basic but important remarks about the Galois depth of birational maps.
In Section~\ref{section:Relations} we study relations of $\BirMori(X)$ and prove Theorem~\ref{theorem:GeneratingRelations}.
Then, we make a detour to Galois theory in Section~\ref{section:GaloisTheory} to establish that a perfect field $\kk$ with $[\overbar\kk:\kk]>2$ has arbitrarily large orbits (Lemma~\ref{lemma:LargeGaloisOrbitExists}), and that there are such fields that do not have an orbit of size exactly $8$ (Lemma~\ref{lemma:FieldWithouthSomeOrbit}).
The latter is to contrast our result with \cite{LZ19}.
The main part of this paper ends with Section~\ref{section:GroupHomo}, where we prove Theorem~\ref{theorem:GroupHomomorphism} and finally Theorem~\ref{theorem:CremonaNotSimple}, concluding the main part of the paper.

In Section~\ref{section:RationalMoriCB}, we study rational Mori conic bundles and give the proof of Theorem~\ref{theorem:Refinement}.
As we use here the long list of Sarkisov links from~\cite[Theorem 2.6]{Iskovskikh96}, we provide the reader with a visualisation of it in Section~\ref{section:longlist}.

\bigskip

I would like to thank Philipp Habegger and Lars Kuehne for discussions about Galois theory.
Moreover, I thank Susanna Zimmermann for explaining her results to me, and my thesis advisor J\'er\'emy Blanc for introducing me to the Cremona group.

\section{Preliminaries}\label{section:CremonaOverPerfectPreliminaries}

Consider a perfect field $\kk$ and the Galois group $\Gamma=\Gal(\overbar\kk/\kk)=\Aut(\overbar\kk/\kk)$.
We will work over the algebraic closure $\overbar\kk$, equipped with the Galois action of $\Gamma$.
A perfect field is a field such that every algebraic extension is separable.
We will use the following property of perfect fields:
A field $\kk$ is perfect if and only if the extension $\overbar\kk/\kk$ is normal and separable, which means that $\overbar\kk/\kk$ is Galois.
In particular, the field fixed by the action of $\Gamma$ is exactly $\kk$ \cite[Theorem 1.2, Chapter VI]{lang05}.
So a point is fixed by the Galois action if and only if it is defined over $\kk$.

We are interested in surfaces.
\textbf{In the sequel, we assume all surfaces to be smooth and projective.}
A (rational) map $\varphi\colon X\dashto Y$ is always supposed to be defined over $\kk$ (and thus $X$ and $Y$ are defined over $\kk$, too).
However, we will look at $\overbar\kk$-points and $\overbar\kk$-curves on our surfaces.

\subsection{Spaces of interest: Mori fiber spaces}

\begin{definition}\label{def:RelativePicGroup}
  Let $X$ be a surface and $\pi\colon X\to B$ a surjective morphism to a smooth variety $B$. The \textit{relative Picard group} is the quotient $\Pic_\kk(X/B):=\Pic_\kk(X)/\pi^*\Pic_\kk(B)$.
\end{definition}

We study Mori fiber spaces and only consider surfaces $X$ over $\kk$. So the definition is as follows:

\begin{definition}\label{def:MoriFiberSpace}
  A surjective morphism $\pi\from X\to B$ with connected fibers, where $X$ is a smooth surface and $B$ is smooth, is called a \textit{Mori fiber space} if the following conditions are satisfied:
  \begin{enumerate}
    \item
    $\dim(B)<\dim(X)$,
    \item
    $\rk\Pic_\kk(X/B)=1$ (relative Picard rank),
    \item
    $-K_X\cdot D>0$ for all effective curves $D$ on $X$ that are contracted by $\pi$.
  \end{enumerate}
  If $B$ is $1$-dimensional, we say that $X\to B$ is a \textit{Mori conic bundle}.
\end{definition}

Note that we say that an arbitrary fibration (not necessarily a Mori fiber space) $X\to B$ with $B$ one-dimensional is a \textit{conic bundle}, if a general fiber is isomorphic to $\PPP^1(\bar\kk)$ and any singular fiber is the union of two $(-1)$-curves intersecting at one point.

\begin{remark}\label{remark:MFSToPointOrCurve}
  Any Mori fiber space $X\to B$ is of one of the two following forms, depending on the dimension of the base $B$:
  \begin{enumerate}
    \item
    If $B=\{\asterisk\}$ then $\rk\Pic_\kk(X)=1$ and so $X$ is a del Pezzo surface.
    \item
    If $B$ is a curve, then $X\to B$ is a Mori conic bundle, and the fiber of each $\overbar\kk$-point of $B$ is isomorphic to a reduced conic in $\PPP^2$ (irreducible or reducible).
    In particular, it is a conic bundle.
    Moreover, the two irreducible components of any singular fiber lie in the same Galois orbit.
    (Otherwise, the orbit of one component consists of disjoint $(-1)$-curves.
    Contracting them yields a surface $S\to B$ with $\rk(\Pic(S/B))\geq 1$.
    Then $\rk(\Pic_\kk(X/B))=\rk(\Pic_\kk(S/B))+1\geq 2$, a contradiction.)
  \end{enumerate}
  If $B$ is $1$-dimensional and $X$ is geometrically rational with a $\kk$-point, the following lemma implies that $B=\PPP^1$ and so $\rk\Pic_\kk(X)=2$.
\end{remark}

\begin{lemma}\label{lemma:GeometricallyRationalIffGenus0}
  Let $\pi\colon X\to B$ be a conic bundle.
  Then $X$ is geometrically rational if and only if the genus of $B$ is $0$.
  In this case, either $B=\PPP^1$ or $B$ is a smooth projective curve without $\kk$-point.
\end{lemma}

\begin{proof}
  Assume that $X$ is geometrically rational, so there is a birational map $\varphi\colon X\dashto \PPP^2$ defined over $\overbar\kk$. Hence, $\pi\varphi^{-1}\colon\PPP^2\dashto B$ is a dominant rational map, defined over $\overbar\kk$, and so $B$ is geometrically unirational.
  As $B$ is a curve, the solution of L\"uroth's problem implies that $B$ is geometrically rational. So the genus of $B$ is zero.

  The converse direction is a corollary of Tsen's theorem \cite[Corollary 6.6.2, p. 232]{kollar99}, which states that there is a birational map from $X$ to $B\times\PPP^1$ defined over $\overbar\kk$ and hence $X$ is geometrically rational.
\end{proof}

\begin{remark}\label{remark:ConicBundleOverP1}
  For a conic bundle $X\to\PPP^1$, there are $8-K_X^2$ singular fibers on $X$.
  (Over $\bar\kk$, contracting one irreducible component of each singular fiber gives a morphism $X\to Y$, where $Y$ is a $\PPP^1$-bundle over $\PPP^1$, that is, $Y$ is a Hirzebruch surface.
  Hence, $K_X^2=K_Y^2-r=8-r$, where $r$ is the number of singular fibers on $X$.)
\end{remark}

\begin{definition}\label{definition:PreserveFibrationAndIsoOfMFS}
  Let $X_1\to B_1$ and $X_2\to B_2$ be two Mori fiber spaces.
  We say that a birational map $\varphi\colon X_1\dashto X_2$ \textit{preserves the fibration} if the diagram
  \begin{center}
    \begin{tikzpicture}
      \matrix (m) [matrix of math nodes,row sep=1em,column sep=1em,minimum width=2em]
      {
          X_1 & & X_2  \\
          B_1 & & B_2 \\};
          \path[-stealth]
          (m-1-1) edge[dashed] node[above] {$\varphi$} (m-1-3)
          (m-1-1) edge node {} (m-2-1)
          (m-2-1) edge[white] node[black] {$\simeq$} (m-2-3)
          (m-1-3) edge node {} (m-2-3)
          ;
        \end{tikzpicture}
  \end{center} commutes.
  Moreover, if $\varphi$ is also an isomorphism we say that $\varphi\colon X_1\simto X_2$ is an \textit{isomorphism of Mori fiber spaces}.
\end{definition}

\subsection{Maps of interest: Sarkisov links}

\begin{definition}\label{def:BirMori}
  Let $X\to B$ be a Mori fiber space.
  We denote by $\BirMori(X)$ the groupoid consisting of all birational maps $\varphi\colon X_1\dashto X_2$ where $X_i\to B_i$ are Mori fiber spaces for $i=1,2$ such that $X_1$ and $X_2$ are birational to $X$.
\end{definition}

\begin{definition}\label{def:RelativeBir}
  For a Mori fiber space $X\stackrel{\pi}\to W$ we denote by $\Bir(X/W)\subset\Bir(X)$ the subgroup of birational maps $f\in\Bir(X)$ that preserve the fibration.
\end{definition}

\begin{definition}\label{def:SarkisovLink}
  A \textit{Sarkisov link} (or simply \textit{link}) is a birational map $\varphi\colon X_1\dashto X_2$ between two Mori fibrations $\pi_i\colon X_i\to B_i$, $i=1,2$, that is part of a commuting diagram of one of the following four types:%
  \begin{center}
    \begin{tabular}{lp{4cm}p{6cm}}
      Type I\vspace{0pt} & \centering\begin{tikzpicture}[baseline=0]
       \matrix (m) [matrix of math nodes,row sep=1em,column sep=1em,minimum width=2em]
       {
           &   &  \\
           X_1 &   & X_2 \\
           \{\asterisk\}=B_1 &  & B_2 \\};
       \path[-stealth]
       (m-2-1) edge[dashed] node[above] {$\varphi$} (m-2-3)
       (m-2-1) edge node {} (m-3-1)
       (m-2-3) edge node {} (m-3-3)
       (m-3-3) edge node {} (m-3-1)
       ;
     \end{tikzpicture}%

     & {where $\varphi^{-1}\colon X_2\to X_1$ is the blow up of one orbit of $\Gal(\overbar\kk/\kk)$.}\\
   \end{tabular}
 \end{center}
 \begin{center}
\begin{tabular}{lp{4cm}p{6cm}}
     Type II &\vspace{-20pt}\centering\begin{tikzpicture}
      \matrix (m) [matrix of math nodes,row sep=1em,column sep=1em,minimum width=2em]
      {
          &  Z &  \\
          X_1 &   & X_2 \\
          B_1 &  & B_2 \\};
      \path[-stealth]
      (m-2-1) edge[dashed] node[above] {$\varphi$} (m-2-3)
      (m-1-2) edge node[left] {$\sigma_1$} (m-2-1)
      (m-1-2) edge node[right] {$\sigma_2$} (m-2-3)
      (m-2-1) edge node {} (m-3-1)
      (m-2-3) edge node {} (m-3-3)
      (m-3-1) edge node[above] {$\simeq$}  (m-3-3)
      ;
    \end{tikzpicture}%
    &
    where $\sigma_i\colon Z\to X_i$ is a blow-up of one orbit.\\
  \end{tabular}
\end{center}
\begin{center}
\begin{tabular}{lp{4cm}p{6cm}}
    Type III & \vspace{-20pt}\centering\begin{tikzpicture}
     \matrix (m) [matrix of math nodes,row sep=1em,column sep=1em,minimum width=2em]
     {
         &   &  \\
         X_1 &   & X_2 \\
         B_1 &  & \{\asterisk\}=B_2 \\};
     \path[-stealth]
     (m-2-1) edge[] node[above] {$\varphi$} (m-2-3)
     (m-2-1) edge node {} (m-3-1)
     (m-2-3) edge node {} (m-3-3)
     (m-3-1) edge node[above] {}  (m-3-3)
     ;
   \end{tikzpicture}%
    &
   where $\varphi\colon X_1\to X_2$ is the blow-up of one orbit.\\
 \end{tabular}
\end{center}
\begin{center}
\begin{tabular}{lp{4cm}p{6cm}}
   Type IV &     \vspace{-20pt} \centering\begin{tikzpicture}
         \matrix (m) [matrix of math nodes,row sep=1em,column sep=1em,minimum width=2em]
         {
             &  &  \\
             X_1 &   & X_2 \\
             B_1 &  & B_2 \\};
         \path[-stealth]
         (m-2-1) edge[] node[above] {$\varphi$} node[below] {$\simeq$}(m-2-3)
         (m-2-1) edge node {} (m-3-1)
         (m-2-3) edge node {} (m-3-3)
         ;
       \end{tikzpicture}%
       &
       where $\varphi$ is an isomorphism that does not preserve the fibration and $B_1$, $B_2$ are curves.
    \end{tabular}
  \end{center}

\end{definition}

\begin{remark}
  In fact, $B_1$ in a link of type III (respectively $B_2$ in a link of type I) is always $\PPP^1$ (see \cite{Iskovskikh96}).
  In a link of type IV, $B_1$ is a smooth curve of genus $0$ (follows from the proof of Lemma~\ref{lemma:IskoForNotGeomRational} and Lemma~\ref{lemma:GeometricallyRationalIffGenus0}), but not necessarily $\PPP^1$:
  For example, let $C$ be a conic with no $\kk$-point.
  Then the change of fibration of $C\times C$ is a link of type IV.
\end{remark}

\begin{theorem}[\hspace{1sp}{\cite[Theorem 2.5]{Iskovskikh96} and \cite[Appendix]{Corti95}}]\label{theorem:Iskovskikh}
  Let $X$ be a geometrically rational Mori fiber space.
  The groupoid $\BirMori(X)$ is generated by Sarkisov links and isomorphisms of Mori fiber spaces.
\end{theorem}

The following two lemmas follow from the classification of Sarkisov links \cite[Theorem 2.6]{Iskovskikh96}.
In order to keep the promise of providing elementary proofs, we reprove the statement for completeness.

\begin{lemma}\label{lemma:TypeIIIDelPezzo}
  Let $\varphi\colon X_1\to X_2$ be a link of type III.
  \begin{enumerate}
    \item \label{item:TypeIIIDelPezzo--NE}
    The cone of effective curves $\NE_\QQQ(X_1)$ equals $\QQQ_{\geq 0} f+\QQQ_{\geq 0} E$, where $f$ is a fiber of the Mori fiber space $X_1\to B_1$ and $E$ is the exceptional locus of $\varphi$.
    \item \label{item:TypeIIIDelPezzo--dP}
    $X_1$ is a del Pezzo surface.
    \item \label{item:TypeIIIDelPezzo--8}
    The size $r$ of the orbit blown-up by $\varphi$ is less or equal than $8$.
  \end{enumerate}
\end{lemma}

\begin{proof}
  As $E$ is an orbit of $r$ disjoint $(-1)$-curves, none of these is contained in a fiber of $X_1\to B_1$ (see Remark~\ref{remark:MFSToPointOrCurve}).
  Hence, $f\cdot E>0$ and so $f$ and $E$ are not linearly equivalent because $f^2=0$ and $E^2=-r<0$.
  As the rank of the Picard group $\Pic_\kk(X_1)$ is $2$ (since the rank of the Picard group of $X_2$ is $1$), any curve $C$ in $\NE_\QQQ(X_1)$ is linearly equivalent to $\alpha f+\beta E$ for some $\alpha,\beta\in\QQQ$.
  We want to prove that $\alpha,\beta\geq 0$.
  If $C=E$, then $\alpha=0$ and $\beta=1$, so assume that $C\neq E$.
  We compute
  \[
    0\leq f\cdot C=\beta E\cdot f,
  \]
  hence $\beta\geq 0$.
  We also find
  \[
    0\leq E\cdot C = \alpha E\cdot f+\beta E^2,
  \]
  which implies that $\alpha\geq 0$, since $\beta E^2\leq 0$ and $E\cdot f>0$.
  This proves~\ref{item:TypeIIIDelPezzo--NE}.

  For~\ref{item:TypeIIIDelPezzo--dP}, we prove ampleness of $-K_{X_1}$ using Kleiman's Criterion.
  Note that the expression of $\NE_\QQQ(X_1)$ in~\ref{item:TypeIIIDelPezzo--NE} is closed, hence $\overbar{\NE_\QQQ(X_1)}\setminus\{0\}=\left(\QQQ_{\geq 0} f+\QQQ_{\geq 0} E\right)\setminus\{0\}$.
  We compute for $(\alpha,\beta)\in\QQQ_{\geq0}^2\setminus\{(0,0)\}$
  \[
    -K_{X_1}(\alpha f+\beta E)=\alpha(-K_{X_1}f)+\beta (-K_{X_1}E)=2\alpha +\beta r>0,
  \]
  where we used the adjunction formula to compute $-K_{X_1}f$ and $-K_{X_1}E$.
  Therefore, Kleiman's Criterion implies that $-K_{X_1}$ is ample, and so~\ref{item:TypeIIIDelPezzo--dP} holds.

  For~\ref{item:TypeIIIDelPezzo--8}, note that~\ref{item:TypeIIIDelPezzo--dP} implies in particular that $0<(-K_{X_1})^2\leq 9$.
  Since $\varphi\colon X_1\to X_2$ is a blow-up of $r$ points (over $\overbar\kk$), we also have $(-K_{X_1})^2=(-K_{X_2})^2-r$.
  This gives $r<(-K_{X_1})^2\leq 9$, hence $r\leq 8$.
\end{proof}

\begin{lemma}\label{lemma:TypeIIDelPezzo}
  For $i=1,2$, let $X_i\to\{\asterisk\}$ be two Mori fiber spaces and let $\varphi\colon X_1\dashto X_2$ be a link of type II that has a resolution $\sigma_i\colon Z\to X_i$,  where $\sigma_i$ is the blow-up of an orbit of size $r_i$ with exceptional divisor $E_i$.
  \begin{enumerate}
    \item\label{item:TypeIIDelPezzo--EffectiveCone}
    The cone of effective curves $\NE_\QQQ(Z)$ equals $\QQQ_{\geq0}E_1+\QQQ_{\geq 0}E_2$.
    \item\label{item:TypeIIDelPezzo--SmallOribt} $Z$ is a del Pezzo surface.
    In particular, $r_i\leq 8$ for $i=1,2$.
  \end{enumerate}
\end{lemma}

\begin{proof}
  For~\ref{item:TypeIIDelPezzo--EffectiveCone} it is enough to show that $E_1$ and $E_2$ are (different) extremal rays in the cone of effective $1$-cycles $\NE(Z)$ since $\rk\Pic_\kk(Z)=2$.
  First, we remark that $\NE_\QQQ(X_1)=-K_{X_1}\cdot\QQQ_{\geq0}$:
  Having $\rk\Pic(X_1)=1$, there is a curve $C$ on $X_1$ such that $\NE_\QQQ(X_1)=\QQQ_{\geq0}C$.
  As $-K_{X_1}$ is ample (since $X_1$ is del Pezzo), $-K_{X_1}$ is effective and non-zero, hence $-K_{X_1}=\lambda C$ for some $\lambda>0$.
  Therefore, $\NE_\QQQ(X_1)=-K_{X_1}\QQQ_{\geq0}$.
  Now, let $D\in\NE_\QQQ(Z)$ be such that $D$ is no multiple of $E_1$.
  Hence, $\pi_*(D)$ is effective.
  As $D\in\Pic(Z)=\QQQ E_1+\QQQ\pi^*(-K_{X_1})$, we can write $D\sim a\pi^*(-K_{X_1})+bE_1$ with $a\geq 0$.
  Therefore, $\NE_\QQQ(Z)$ lies in $\QQQ_{\geq0}\pi^*(-K_{X_1})+\QQQ E_1$.
  Hence, $E_1$ is extremal. The same argument works for $E_2$.
  The two extremal rays $E_1$ and $E_2$ are different because $E_1$ is effective with $E_1^2<0$ but $E_1E_2\geq 0$ ($\varphi$ is not an isomorphism, hence $E_1$ and $E_2$ are distinct).

  For~\ref{item:TypeIIDelPezzo--SmallOribt}, compute with the adjunction formula $-K_ZE_i=r_i$.
  For $\alpha,\beta \in\QQQ_{\geq0}$, not both zero, this gives \[-K_Z(\alpha E_1+\beta E_2)=\alpha r_1+\beta r_2>0\] and Kleiman criterion implies that $-K_Z$ is ample, hence $Z$ is del Pezzo.
  Note that $r_i\leq 8$ follows in the same way as the proof of~\ref{item:TypeIIIDelPezzo--8} of Lemma~\ref{lemma:TypeIIIDelPezzo}.
\end{proof}

This leads us to one of the main points of this article: The \textit{Galois depth} of birational maps, which plays in our article the role of the covering gonality in \cite[Theorem D]{BLZ19}.

\subsection{Galois depth of birational maps}

\begin{definition}\label{def:Galois depth}
  Let $\varphi\colon X\dashto X'$ be a birational map between surfaces. The \textit{Galois depth} of $\varphi$ is the maximal size of all orbits contained in the base loci of $\varphi$ or $\varphi^{-1}$.
\end{definition}

\begin{remark} \label{remark:SmallGalois depth}
  A Sarkisov link $\varphi\colon X_1\dashto X_2$ has Galois depth $\leq 8$ -- except if it is a link of type II between Mori conic bundles.
  Indeed, the statement for links of type I and III is implied by Lemma~\ref{lemma:TypeIIIDelPezzo}, for type II it is Lemma~\ref{lemma:TypeIIDelPezzo}, and links of type IV do not have base points so they have Galois depth $0$.
\end{remark}

\begin{remark}\label{remark:DecompositionOfBirMaps}
  With the above remark, any $f\in\BirMori(X)$ can be decomposed by Iskovskikh's Theorem~\ref{theorem:Iskovskikh} as
  \[f = \Psi_{N+1}\circ\Phi_N\circ\Psi_N\circ\cdots\circ\Psi_2\circ\Phi_1\circ\Psi_1,\]
  where $\Psi_i\colon Y_{i-1}\dashrightarrow X_i$ are birational maps with Galois depth $\leq 8$ (or an isomorphism of Mori fiber spaces), and $\Phi_i\colon X_i\dashrightarrow Y_i$ are a composition of links of type II between Mori conic bundles, possibly followed by an isomorphism of Mori fiber spaces.

  Note that if $f$ is not an isomorphism of Mori fiber spaces, then at least one Sarkisov link has to appear in the above decomposition.
\end{remark}

\section{Relations}\label{section:Relations}

Let $X$ be a Mori fiber space (of dimension $2$).
By Theorem~\ref{theorem:Iskovskikh}, the groupoid $\BirMori(X)$ is generated by Sarkisov links and isomorphisms of Mori fiber spaces.
We study the set of relations of these generators.
The following two relations will be called \textit{trivial}:
\begin{itemize}
  \item
  $\alpha\beta=\gamma$, where $\alpha,\beta,\gamma$ are isomorphisms of Mori fiber spaces,
  \item
  $\alpha\psi^{-1}\varphi=\id_X$, where $\varphi\colon X\dashto Y$ and $\psi\colon Z\dashto Y$ are Sarkisov links and $\alpha\colon Z\simto X$ is an isomorphism of Mori fiber spaces.
\end{itemize}

Whereas \cite{LZ19} and \cite{BLZ19} encode the information of Sarkisov links and relations between them in a simplicial square complex (where the vertices are rank $r$-fibrations) and find generating relations via the Sarkisov program \cite[Sections 2 and 3]{LZ19}, here we discuss a more classical and elementary approach.

\subsection{Relations between Mori conic bundles}\label{subsection:RelationsConicBundles}

\begin{lemma}\label{lemma:ContractedCurveInFiberMeansTypeII}
  Let $X\to V$ and $Y\to W$ be two Mori conic bundles and let $\varphi\colon X\dashto Y$ be a birational map such that every curve contracted by $\varphi$ is contained in a fiber of $X\to V$.
  Then there is a composition $\varphi=\alpha\circ\varphi_n\circ\cdots\circ \varphi_1$ of Sarkisov links $\varphi_i$ of type II between Mori conic bundles such that each $\varphi_i$ is the  blow-up of one orbit of $r_i\geq 1$ distinct points on $r_i$ distinct smooth fibers, followed by the contraction of the strict transforms of these fibers, where $\alpha$ is an isomorphism (not necessarily of Mori fiber spaces).
\end{lemma}

\begin{proof}
  Consider the minimal resolution \begin{center}
  \begin{tikzpicture}
    \matrix (m) [matrix of math nodes,row sep=1em,column sep=1em,minimum width=2em]
    {
    & S &  \\
    X & & Y ,  \\
    };
    \path[-stealth]
    (m-2-1) edge[dashed] node[below] {$\varphi$} (m-2-3)
    (m-1-2) edge node[above] {$\sigma$} (m-2-1)
    (m-1-2) edge node[above] {$\tau$} (m-2-3)
    ;
    \end{tikzpicture}%
  \end{center}
  where $\sigma$ and $\tau$ are compositions of blow-ups in points.
  Let $C\subset S$ be a $(-1)$-curve contracted by $\tau$.
  So $\sigma(C)$ is either a smooth fiber or a component of a singular fiber.

  Let us show that it is not possible that $\sigma(C)=F$ is a component of a singular fiber.
  In this case, the self-intersection of $F$ is $-1$.
  Hence, no point on $F$ is a base point of $\sigma$, since otherwise $C^2\leq -2$.
  Let $E$ be the other irreducible component of the fiber containing $F$, so $E^2=-1$.
  By Remark~\ref{remark:MFSToPointOrCurve}, $E$ and $F$ lie in one orbit.
  As $\tau$ contracts $C=\tilde E$, it also contracts $\tilde F$ (since they are in the same orbit).
  This is not possible, since after the contraction of $C$ onto a point, the push-forward of $\tilde F$ is a curve of self-intersection $0$ and can therefore not be contracted.

  Hence, $\sigma(C)=F$ is a smooth fiber and so its self-intersection is $0$.
  As $\sigma$ is a composition of blow-ups, there is one base point $p\in F$ of $\sigma$.
  So all points of the Galois orbit of $p$ are base points of $\sigma$, and no two lie on the same fiber, since otherwise the self-intersection of $C$ would be $\leq -2$.
  Therefore, $\tau$ contracts all fibers through the orbit.
  Let $\varphi_1$ be the blow-up of the orbit of $p$ followed by the contraction of the strict transforms of the fibers through the orbit.
  This is a link of type II between Mori conic bundles, and $\varphi$ factors through $\varphi_1$.
  Moreover, $\varphi\circ\varphi^{-1}$ has fewer base points.

  Repeating this process for all $(-1)$-curves that are contracted by $\tau$ gives an isomorphism $\alpha=\varphi\circ\varphi_1^{-1}\circ\cdots\circ\varphi_n^{-1}$, that is, $\varphi=\alpha\circ\varphi_n\circ\cdots\circ\varphi_1$ where all $\varphi_i$ are blow-ups of an orbit followed by the contraction of the strict transforms of the fibers through it, as in the statement of the lemma.
\end{proof}

\begin{corollary}\label{corollary:PreservingTheFibrationMeansTypeII}
  Let $X\to V$ and $Y\to W$ be two Mori fiber spaces and let $\varphi\colon X\dashto Y$ be a birational map that preserves the fibration.
  If $\varphi$ is not an isomorphism, then $\varphi=\varphi_n\circ\cdots\circ \varphi_1$ for Sarkisov links $\varphi_i$ of type II between Mori conic bundles such that each $\varphi_i$ is the  blow-up of one orbit of $r_i\geq 1$ distinct points on $r_i$ distinct smooth fibers, followed by the contraction of the strict transforms of these fibers.
\end{corollary}

\begin{proof}
  Since $\varphi$ preserves the fibration, any curve in $X$ that is contracted by $\varphi$ onto a point in $Y$ is contained in a fiber and so Lemma~\ref{lemma:ContractedCurveInFiberMeansTypeII} can be applied.
  Hence, $\varphi=\alpha\circ\varphi_n\circ\cdots\circ \varphi_1$ where the $\varphi_i$ are links of type II as desired and $\alpha$ is an isomorphism.
  As $\varphi$ and each of the $\varphi_i$ preserve the fibration, also $\alpha$ preserves the fibration and is therefore an isomorphism of Mori fiber spaces. Hence $\alpha\circ\varphi_n$ is also a link of type II between Mori conic bundles and the corollary follows.
\end{proof}

Note that Corollary~\ref{corollary:PreservingTheFibrationMeansTypeII} implies that a birational map preserving the fibration sends singular fibers onto singular fibers.

\begin{lemma}\label{lemma:IskoForNotGeomRational}
  Let $X$ be a Mori fiber space that is not geometrically rational. Then, $\BirMori(X)$ is generated  by links of type II and isomorphisms of Mori fiber spaces.
\end{lemma}

\begin{proof}
  Let $\varphi\colon X_1\dashto X_2$ be a birational map in $\BirMori(X)$, where $X_i\to B_i$ are Mori fiber spaces where $B_i$ are curves of genus at least $1$.
  Any curve $C$ in $X_1$ contracted by $\varphi$ is rational, and so every morphism $C\to B_1$ is constant (by Riemann-Hurwitz).
  Therefore, $C$ is contained in a fiber.
  Lemma~\ref{lemma:ContractedCurveInFiberMeansTypeII} can be applied and therefore $\varphi=\alpha\varphi_n\cdots\varphi_1$, where $\alpha$ is an isomorphism and $\varphi_i$ are links of type II between Mori conic bundles.
  If $\alpha$ preserves the fibration, then we are in the case of Corollary~\ref{corollary:PreservingTheFibrationMeansTypeII} and so $\varphi$ is a product of links of type II.
  If $\alpha$ does not preserve the fibration, $\alpha\colon Y_1\simto Y_2$ is a link of type IV, where $Y_i\to C_i$ are Mori fiber spaces. Such a link does not exist for geometrically non-rational surfaces:
  Consider $F_2\subset Y_2$ a general fiber (hence isomorphic to $\PPP^1$) of $Y_2\to C_2$ and consider its image $F_1=\alpha^{-1}(F_2)$ in $Y_1$, which is also isomorphic to $\PPP^1$.
  As $F_2$ was chosen to be a general fiber and $\alpha$ does not preserve the fibration, the restriction of the fibration $Y_1\to C_1$ to $F_1$ is surjective. This gives a contradiction, since every map from $\PPP^1\to C_1$ is constant as before.
\end{proof}

\begin{remark}\label{remark:RelationOfFourLinks}
    Let $\chi_1\colon X_0\dashto X_1$ and $\chi_2\colon X_1\dashto X_2$ be two links of type II between Mori conic bundles $X_i\to B_i$ such that no fiber of $X_1\to B_1$ is contracted by both $\chi_1^{-1}$ and $\chi_2$.
    Then, the composition $\chi_2\chi_1$ can be written as $\chi_3^{-1}\chi_4^{-1}$, where $\chi_3\colon X_2\dashto X_3$ is the blow-up of $\chi_2(\Bas(\chi_1^{-1}))$ followed by the contraction of the strict transform of the corresponding fiber, and $\chi_4\colon X_3\dashto X_4$ is the blow-up of $\chi_3(\Bas(\chi_2^{-1}))$ followed by the contraction of the strict transform of the corresponding fiber.
    So $\chi_4\chi_3\chi_2\chi_1$ is an isomorphism of Mori fiber spaces and $\chi_4$ can be chosen such that $\chi_4\chi_3\chi_2\chi_1=\id$.

    Note that $\chi_1$ and $\chi_3$ have the same Galois depth, as well as $\chi_2$ and $\chi_4$.
\end{remark}

\begin{lemma}\label{lemma:RelationsGeneratedByFour}
  Let $X$ be a Mori conic bundle and let $\varphi_n\cdots\varphi_1=\id_X$ be a relation in $\BirMori(X)$ such that $\varphi_i\colon X_{i-1}\dashto X_i$ is a link of type II between Mori conic bundles, where $X_0=X_n=X$.
  This relation is generated in $\BirMori(X)$ by the trivial relations and those of the form $\chi_4\chi_3\chi_2\chi_1=\id$ as in Remark~\ref{remark:RelationOfFourLinks}, where $\chi_1,\ldots,\chi_4$ are links of type II between Mori conic bundles.
\end{lemma}

\begin{proof}
  For each link $\varphi_i$ of type II we call $p_{i-1}\subset X_{i-1}$ the base orbit of $\varphi_i$, and $q_i\subset X_i$ the base orbit of $\varphi_i^{-1}$.

  In the following, we show that using relations of the form $\chi_4\chi_3\chi_2\chi_1=\id$ and the trivial relations, we can reduce the word $\varphi=\varphi_n\cdots\varphi_1$ to the empty word.

  Starting from a fiber $F=F_0\subset X_0$ and the value $\NNNN(F,0)=0$, we define a sequence of subsets $F_i=(\varphi_i\cdots\varphi_1)(F)\subset X_i$, and a sequence of values $\NNNN(F,i)\in\NNN$ for $i=1,\ldots,n$ that ``keep track of what happens to $F$'' in each step of our fixed decomposition $\varphi_n\circ\cdots\circ\varphi_1$.

  We inductively define $\NNNN(F,i)\geq0$ in the following way:
  \begin{enumerate}
    \item
    If $\varphi_i$ is a local isomorphism on the fiber containing $F_{i-1}$, then
    \[
    \NNNN(F,i)=\NNNN(F,i-1).
    \]
    \item
    Otherwise, $F_{i-1}$ lies on the same fiber as a point of $p_{i-1}$. We define
    \[
    \NNNN(F,i)=
    \begin{cases}
      1 & \text{if $F_{i-1}$ is a fiber (so $F_i$ is a point in $q_i$),}\\
      \NNNN(F,i-1)-1 & \text{if $F_{i-1}$ is a point in $p_{i-1}$,}\\
      \NNNN(F,i-1)+1 & \text{if $F_{i-1}$ is a point but not in $p_{i-1}$}\\
      & (\text{again, $F_i$ is a point in $q_i$)}.
    \end{cases}
      \]
  \end{enumerate}
  Observe that $\NNNN(F,i)$ is the number of base points of $\varphi_i\circ\cdots\circ\varphi_1$ that are equal or infinitely near to a base point on $F$.
  Note that the sequence
  \[
  \Sigma_F = \left(\NNNN(F,0),\NNNN(F,1),\ldots,\NNNN(F,n)\right)
  \]
  is the same for each fiber in the same orbit as $F$.
  We consider connected subsequences of $\Sigma_F$ and note that the last value $\NNNN(F,n)$ is zero as the product of all the $\varphi_i$ is the identity.
  \begin{enumerate}
    \item \label{item:RelationsGeneratedByFour--updown}
    First, we look at subsequences of the form $(m-1,m,m-1)$ for $m\geq1$ with corresponding links $\varphi_i$ and $\varphi_{i+1}$.
    This occurs only if $\varphi_i^{-1}$ and $\varphi_{i+1}$ have common base points, so $\varphi_{i+1}\varphi_i$ equals an isomorphism.
    Hence, this part of the sequence is equivalent to the empty word modulo the trivial relations.
    \item \label{item:RelationsGeneratedByFour--upconstant}
    Now, we consider subsequences of the form $(m-1,m,m)$ with corresponding links $\varphi_i$, which is not an isomorphism on $F$, and $\varphi_{i+1}$, which is a local isomorphism on $F$.
    In this case, the fibers of $X_i\to B_i$ that are contracted by $\varphi_i^{-1}$ are not contracted by $\varphi_{i+1}$.
    By Remark~\ref{remark:RelationOfFourLinks}, there exist links $\chi_i$ and $\chi_{i+1}$ of type II between Mori conic bundles such that $\chi_i$ has $\varphi_{i+1}(\Bas(\varphi_i^{-1}))$ as base points,
    and $\chi_{i+1}$ has base points $\chi_i(\Bas(\varphi_{i+1}))$, and $\varphi_{i+1}\varphi_i=\chi_i^{-1}\chi_{i+1}^{-1}$ is satisfied.
    So by replacing $\varphi_{i+1}\varphi_i$ with $\chi_i^{-1}\chi_{i+1}^{-1}$ in the factorisation of $\varphi$, we can change this part of the sequence to $(m-1,m-1,m)$, leaving the rest of it invariant.
  \end{enumerate}
  Using these two kinds of reduction modulo the said relations, we proceed by induction over
  \[
  m=m_F=\max\{\NNNN(F,i)\mid i=0,\ldots,n\}.
  \]
  There exists at least one subsequence $\Sigma'= (m-1,m,\ldots,m,m-1)$ of size $k+2$ for some $k\geq 1$.
  Using~\ref{item:RelationsGeneratedByFour--upconstant} multiple times we can change this part of the sequence to $(m-1,\ldots,m-1,m,m-1)$.
  This can then be reduced to $(m-1,\ldots,m-1)$ of size $k+1$ with~\ref{item:RelationsGeneratedByFour--updown}.
  Doing this for any such $\Sigma'$, we get a new sequence (corresponding to the new factorisation of $\varphi$) whose maximum is $m-1$.
  By induction, we find the sequence $(0,\ldots,0)$.
  Hence, $\varphi$ is a local isomorphism on $F$.

  Note that $m_F\geq 1$ only for finitely many fibers $F$.
  We repeat the described process for each fiber $F$ with $m_F\geq 1$ and can therefore reduce $\varphi$ to an isomorphism using the trivial relations and compositions of four links of type II between Mori conic bundles.
\end{proof}

\begin{corollary}\label{corollary:MoveSmallGalois depthToEnd}
  Let $\Delta\geq 1$.
  Let $X_1\to B_1$ and $X_2\to B_2$ be two Mori conic bundles. Let $\varphi=\varphi_n\cdots\varphi_1\colon X_1\dashto X_2$ be a composition of links of type II.
  Then we can write $\varphi$ modulo the trivial relations and those of the form $\chi_4\chi_3\chi_2\chi_1$, where $\chi_i$ are links of type II between Mori conic bundles, as $\psi_n\cdots\psi_1$, where $\psi_1,\ldots,\psi_r$ are of Galois depth $\geq \Delta$ and $\psi_{r+1},\ldots,\psi_n$ are of Galois depth $<\Delta$.
\end{corollary}

\begin{proof}
  We use relations of Lemma~\ref{lemma:RelationsGeneratedByFour}.
  If $\varphi_i$ and $\varphi_{i+1}$ are centered at the same fiber (that is, $\Bas(\varphi_i^{-1})$ lies on the same orbit of fibers as $\Bas(\varphi_{i+1})$), then they have the same Galois depth.
  If $\varphi_i$ and $\varphi_{i+1}$ are centered at different fibers and the Galois depth of $\varphi_i$ is $<\Delta$ and the one of $\varphi_{i+1}$ is $\geq\Delta$, let $\chi_1=\varphi_i$ and $\chi_2=\varphi_{i+1}$.
  As in Remark~\ref{remark:RelationOfFourLinks} there are links $\chi_3$ (of the same Galois depth as $\chi_1$, hence $<\Delta$) and $\chi_4$ (of the same Galois depth as $\chi_2$, hence $\geq\Delta$) such that $\chi_4\chi_3\chi_2\chi_1=\id$.
  Therefore, by replacing $\chi_2\chi_1$ with $\chi_3^{-1}\chi_4^{-1}$ we have the desired order of orbit sizes on these two elements.
  In this manner we can move all small orbits to the end of the composition.
\end{proof}

\begin{lemma}\label{lemma:NotGeometricallyRational}
  Let $X\to B$ be a Mori conic bundle that is not geometrically rational.
  Relations of the groupoid $\BirMori(X)$ are generated by relations of the form $\chi_4\chi_3\chi_2\chi_1=\id$, where  $\chi_i\colon X_{i-1}\dashto X_i$ are links of type II between Mori conic bundles.
\end{lemma}

\begin{proof}
  By Lemma~\ref{lemma:GeometricallyRationalIffGenus0}, the genus of $B$ is $\geq 1$.
  Links of type I or III do not occur as they are between del Pezzo surfaces (see Lemma~\ref{lemma:TypeIIIDelPezzo}), which are geometrically rational. Also no links of type IV are possible by Lemma~\ref{lemma:IskoForNotGeomRational}.
  Therefore, only links of type II between Mori conic bundles occur.
  By Lemma~\ref{lemma:RelationsGeneratedByFour}, they are generated by links as in the statement.
\end{proof}

The following section is dedicated to the more interesting case of geometrically rational Mori fiber spaces.
We want to prove that relations are generated by this type of relations together with relations of the form $\varphi_n\circ\cdots\circ\varphi_1=\id$, where the Galois depth of all $\varphi_i$ is $\leq 15$.


\subsection{Birational maps on geometrically rational Mori conic bundles}

\begin{lemma}\label{lemma:DivisorsOnGeometricallyRationalCB}
  Let $X\to B$ be a geometrically rational Mori conic bundle with general fiber $f$.
  Then, each effective divisor $D$ on $X$ is linearly equivalent to $-aK_X+bf$ for $a,b\in\frac{1}{2}\ZZZ$ and $a\geq0$.
  Moreover, if the support of $D$ is not contained in fibers, then $a\geq \frac{1}{2}$.
\end{lemma}

\begin{proof}
  By the adjunction formula, we have $-K_X\cdot f=2$. Hence, $-K_X$ and $f$ are linearly independent as $f^2=0$.
  By the geometrical rationality of $X$, we have that $B_{\overbar\kk}=\PPP^1_{\overbar\kk}$, hence the Picard rank of $X$ is $2$. Hence, there are $a,b\in\QQQ$ with $D\sim -aK_X+bf$.
  As $D$ is effective, we have $0\leq D\cdot f=2a\in\NNN$.
  Moreover, if the support of $D$ is not contained in fibers, $a>0$.
  So $a\geq 0$ and $a\in\ZZZ\frac{1}{2}$.
  Since $X$ is geometrically rational, there exists a section $s$ on $X$ defined over $\overbar\kk$ (by a corollary to Tsen's theorem\cite[Corollary 6.6.2, p.232]{kollar99}).
  As $-K_X\cdot s$ is an integer and $D\cdot s=a(-K_X\cdot s)+b$, we also find that $b\in\frac{1}{2}\ZZZ$.
\end{proof}


\begin{lemma}\label{lemma:InverseSarkisovProgram}
  Let $X_1\to W_1$ and $X_2\to W_2$ be two Mori conic bundles that are geometrically rational and let $\varphi\colon X_1\dashto X_2$ be a birational map that preserves the fibration, that is the diagram
  \begin{center} \begin{tikzpicture}
    \matrix (m) [matrix of math nodes,row sep=1em,column sep=1em,minimum width=2em]
    {
        X_1 & & X_2  \\
        W_1 & & W_2 \\};
    \path[-stealth]
    (m-1-1) edge[dashed] node[above] {$\varphi$} (m-1-3)
    (m-1-1) edge node {} (m-2-1)
    (m-2-1) edge[white] node[black] {$\simeq$} (m-2-3)
    (m-1-3) edge node {} (m-2-3)
    ;
  \end{tikzpicture}
  \end{center} commutes.
  Let $H_1\sim -\lambda_1K_{X_1}+\nu_1 f$ be a linear system without fixed component on $X$, and let $H_2\sim -\lambda_2K_{X_2}+\nu_2 f$ be the strict transform of $H_1$ on $X_2$. Then,
  \begin{enumerate}
      \item $\lambda_1=\lambda_2=:\lambda$,
      \item $\nu_2 = \nu_1 + \sum \abs{{\omega_1}}\left(\lambda - m_{\omega_1}\right)$, where the sum runs over all orbits ${\omega_1}$ in $\Bas(\varphi)$,
      \item $m_{\omega_2}(H_2)=2\lambda- m_{\omega_1}$ for orbits $\omega_2$ in $\Bas(\varphi^{-1})$ $($and $\omega_1$ the corresponding orbit in $\Bas(\varphi)$$)$,
  \end{enumerate} where $\abs{\omega_i}$ denotes the size of the orbit $\omega_i$, and $m_{\omega_i}=m_{\omega_i}(H_i)$ denotes the multiplicity of $H_i$ at the points in ${\omega_i}$, for $i=1,2$.
\end{lemma}

\begin{proof}
  Consider the minimal resolution \begin{center}
      \begin{tikzpicture}
        \matrix (m) [matrix of math nodes,row sep=1em,column sep=1em,minimum width=2em]
        {
          & S &  \\
            X_1 & & X_2.  \\
            };
        \path[-stealth]
        (m-2-1) edge[dashed] node[below] {$\varphi$} (m-2-3)
        (m-1-2) edge node[above] {$\sigma_1$} (m-2-1)
        (m-1-2) edge node[above] {$\sigma_2$} (m-2-3)
        ;
      \end{tikzpicture}%
    \end{center}
  We consider the case where $\varphi$ is one link of type II, that is, $\sigma_i$ is the blow-up of one orbit $\omega_i$.
  Let $E_i\subset S$ be the exceptional divisor of the blow-up $\sigma_i$ for $i=1,2$, and let $\hat f$ be a general fiber on $S$.
  So $E_1+E_2=\abs{\omega_1}\hat f$.
  We compute \begin{align*}
      \tilde H_i &= \sigma_i^*(H_i)-m_{\omega_i} E_i\\
      &= -\lambda_i \sigma_i^*(K_{X_i})+\nu_i \sigma_i^*(f)- m_{\omega_i} E_i\\
      &= -\lambda_i (K_S-E_i) +\nu_i \hat f -m_{\omega_i} E_i\\
      & = -\lambda_i K_S +\nu_i \hat f +(\lambda_i -m_{\omega_i})E_i
  \end{align*}
 Replacing $E_1=\abs{\omega_1}\hat f-E_2$ in $\tilde H$ we get \begin{align*}
      \tilde H_1 & = -\lambda_1 K_S+ \left(\nu_1+ \abs{\omega_1}(\lambda_1-m_{\omega_1})\right) \hat f + (m_{\omega_1}-\lambda_1)E_2.
  \end{align*}
  Comparing $\tilde H_1 =\tilde H_2$ we find $\lambda_2=\lambda_1$, $\nu_2=\nu_1 +\abs{\omega_1}(\lambda-m_{\omega_1})$, and $m_{\omega_2}=2\lambda-m_{\omega_1}$.
  Repeating this for every orbit $\omega_1\in \Bas(\varphi)$, we find $\lambda_2=\lambda_1$ and $\nu_2=\nu+\sum \abs{\omega_1}(\lambda-m_{\omega_1})$.
\end{proof}

According to the previous lemma, the behaviour of $\lambda$ and of the multiplicities under a birational map $\Phi$ preserving the fibration is well understood.
The following lemma deals with the situation when $\Phi$ has large Galois depth and is followed by a birational map $\Psi$ that does not preserve the fibration.
We introduce two constants: $\Delta\geq8$, the Galois depth of $\Phi$, and $\delta>0$, a constant that bounds from above the multiplicities of the linear system in the large base orbits with respect to $\lambda$.
The idea is to play with $\Delta$ and $\delta$ to achieve that $\lambda$ increases while the multiplicities do not increase too much (with respect to the ``new'' $\lambda$).
The desired outcome would be $\Delta=9$, but we do not achieve this.
As it is a very technical lemma, the reader might be interested in first reading Corollary~\ref{corollary:NotAnIsoOfConicBundles} and then the proof of Theorem~\ref{theorem:GeneratingRelations} in Section~\ref{section:GeneratingRelations} to see how it is applied.

\begin{lemma}\label{lemma:InductionStepLambda}
  Let $\Delta=16$ and $\delta=\frac{1}{2}$.
  Let $X\to V$, $Y\to W$ and $X'\to V'$ be three minimal Mori conic bundles and let $\Phi\colon X\dashto Y$ be a birational map of Galois depth $\geq \Delta$ and $\Psi\colon Y\dashto X'$ a birational map of Galois depth $<\Delta$ such that
   \begin{center} \begin{tikzpicture}
    \matrix (m) [matrix of math nodes,row sep=1em,column sep=1em,minimum width=2em]
    {
        X & & Y & & X' \\
        V & &W & & V'. \\};
    \path[-stealth]
    (m-1-1) edge[dashed] node[above] {$\Phi$} (m-1-3)
    (m-1-3) edge[dashed] node[above] {$\Psi$} (m-1-5)
    (m-1-1) edge node {} (m-2-1)
    (m-2-1) edge[white] node[black] {$\simeq$} (m-2-3)
    (m-1-3) edge node {} (m-2-3)
    (m-2-3) edge[] node[black] {$\not$} (m-2-5)
    (m-1-5) edge node {} (m-2-5)
    ;
  \end{tikzpicture}%
  \end{center} where we mean by ``~$W\not\to V'$'' that there is no such morphism making the diagram commute.

  Let $H$ be a linear system on $X$ without fixed component, so \[H\sim -\lambda K_X+\nu f,\] where $f$ is a general fiber of $X\to V$, and let $H'\sim -\lambda' K_{X'} + \nu' f'$ be the strict transform of $H$ on $X'$.

  Let $\mu$ $($respectively $\mu'$$)$ be the maximum of the multiplicities $m_\omega(H)$ $($respectively $m_\omega(H')$$)$ for all orbits $\omega$ with $\abs{\omega}\geq\Delta$.

  Assume $\mu<\delta\lambda$.
  Then $\lambda'>\lambda$ and $\mu'<\delta\lambda'$.
\end{lemma}

\begin{proof}
  By Corollary~\ref{corollary:PreservingTheFibrationMeansTypeII}, we can write $\Phi$ as a composition of links of type II between Mori conic bundles.
  By shifting the links of Galois depth $<\Delta$ from $\Phi$ to $\Psi$ via Corollary~\ref{corollary:MoveSmallGalois depthToEnd}, we change the intermediate conic bundle $Y\to W$ without changing $\lambda'$ and $\mu'$.
  So we can assume that each base orbit of $\Phi$ has size $\geq\Delta$.

  Since $\Psi$ is an isomorphism on the points lying in an orbit of size $\geq\Delta$, the maximal multiplicity $\mu'$ on $X'$ of $H'$ equals the multiplicity of $H_Y$ at a point in $Y$, where $H_Y\sim -\lambda_Y K_Y+\nu_Y f_Y$ is the strict transform of $H$ on $Y$.
  Hence, $\mu'\leq H_Y\cdot f_Y=2\lambda_Y$ (when taking $f_Y$ to be the fiber through a point of maximal multiplicity).
  As $\lambda_Y=\lambda$ by Lemma~\ref{lemma:InverseSarkisovProgram}, we get $\mu'\leq 2\lambda$.
  So if we show that $\lambda'>\frac{2}{\delta}\lambda=4\lambda$, then $\delta\lambda'>2\lambda\geq\mu'$ and $\lambda'>\lambda$ are implied.

  Let $g\subset Y$ be the pull back of a general fiber $f'$ of $X'\to V'$ under $\Psi$.
  We write $g\sim -aK_Y+bf$. As $g$ is not a fiber, $a\geq \frac{1}{2}$ by Lemma~\ref{lemma:DivisorsOnGeometricallyRationalCB}.
  We will use $H'\cdot f'=2\lambda'$ to find a lower bound for $\lambda'$.
  Consider a minimal resolution
  \begin{center}
    \begin{tikzpicture}
    \matrix (m) [matrix of math nodes,row sep=1em,column sep=1em,minimum width=2em]
    {
    & T&  \\
    Y & &X'.  \\};
    \path[-stealth]
    (m-2-1) edge[dashed] node[above] {$\Psi$} (m-2-3)
    (m-1-2) edge node {} (m-2-1)
    (m-1-2) edge node {} (m-2-3)
    ;
 \end{tikzpicture}%
\end{center}
  As $f'$ is a general fiber, we have $\tilde H'\cdot \tilde f'=H'\cdot f'$ for the strict transforms of $H'$ respectively $f'$ in $T$.
  So we find \begin{align*}
    2\lambda' & =H'\cdot f'\\
    & = \tilde H_{Y}\cdot \tilde g\\
    & = H_{Y}\cdot g - \sum_j m_{p_j}(H_{Y})m_{p_j}(g),
  \end{align*} where the points $p_j$ are the points blown up in $T\to Y$ (corresponding to points in $Y$ or infinitely near).
  Since $\Psi$ is of Galois depth $<\Delta$, the $p_j$ form orbits of size $<\Delta$.
  As we have remarked in the beginning of the proof, the base points of $\Phi$ consist of orbits of size $\geq \Delta$, hence the map $\Phi$ is a local isomorphism onto the points $p_j$.
  Hence, $M_j:=m_{p_j}(H_{Y})=m_{\Phi^{-1}(p_j)}(H)$.
  We have $H^2-\sum m_q(H)^2\geq0$, where the sum goes over all points $q$ that are either in $X$ or infinitely near (since $H$ is a linear system, hence nef), so also $H^2-\sum M_j^2\geq 0$.
  This gives an upper bound $\sum M_j^2\leq H^2$.

  Let $N_j=m_{p_j}(g)$.
  As $\tilde g^2=\tilde f'^2=0$, we find $g^2=\tilde g^2+\sum N_j^2=\sum N_j^2$.
  By Cauchy-Schwarz, we have $\left(\sum M_jN_j\right)^2\leq (\sum M_j^2)(\sum N_j^2)$ and, with the above discussion, get the inequality
  \begin{align}  \label{eqn:StuffForInequalityForLambda}
    2\lambda' & \geq H_{Y}\cdot g -\sqrt{\left(\sum M_j^2\right)\left(\sum N_j^2\right)}\nonumber\\
    & \geq H_{Y}\cdot g - \sqrt{H^2g^2}.
  \end{align}

  Let $\beta$ be such that $\nu_{Y}=\nu+\beta\lambda$, namely
  \[
    \beta=
  \sum \abs{\omega}(1 - \frac{m_\omega}{\lambda})>\Delta(1-\delta)=8,
  \]
  where the notation is from Lemma~\ref{lemma:InverseSarkisovProgram} and the inequalities come from our assumptions that $\abs{\omega}\geq\Delta=16$ and $m_\omega\leq\mu<\delta\lambda=\frac{1}{2}\lambda$.

  To compare $H_Y\cdot g$ with the square root of $H^2g^2$, let $d=K_Y^2$ and denote by $e_1$ the expression $\frac{1}{\lambda^2}H^2=\lambda(\lambda d+4\nu)=d+4\frac{\nu}{\lambda}$, and similary $e_2=\frac{1}{a^2}g^2=d+4\frac{b}{a}$.
  We compute
  \begin{align*}
    H_{Y}\cdot g& = (-\lambda_YK_Y +\nu_Y f)\cdot(-aK_Y+bf)\\
    &= a\lambda d+2b\lambda +2a\nu_Y\\
    &= \lambda(ad + 2b) +2a(\nu+\beta\lambda)
  \end{align*}
  and so
  \[
  \frac{1}{a\lambda}H_{Y}\cdot g
  = d+2\frac{b}{a} + 2\frac{\nu}{\lambda} +2\beta
  = \frac{1}{2}(e_1+e_2) + 2\beta.
  \]
  Therefore, \begin{align*}
        \frac{2\lambda'}{a\lambda} & \geq \frac{1}{a\lambda}\left(H_{Y}\cdot g - \sqrt{H^2g^2}\right)\\
        &= \frac{e_1}{2}+\frac{e_2}{2} +2\beta - \sqrt{e_1e_2}\\
        &\geq 2\beta,
  \end{align*}
  where the last inequality holds because of the inequality of the arithmetic and the geometric mean.
  Hence, using $a\geq\frac{1}{2}$, we conclude the proof with
  \[
  \lambda' \geq a\beta\lambda>8a\lambda\geq4\lambda.
  \]%
  So we have $\lambda'>\lambda$, and $\delta\lambda'=\frac{1}{2}\lambda'>\frac{1}{2}4\lambda=2\lambda\geq\mu'$.
\end{proof}

\begin{corollary}\label{corollary:NotAnIsoOfConicBundles}
  Let $\Delta=16$. Assume $N\geq 1$.
  For $i=0,\ldots, N$ let $X_i\to V_i$ and $Y_i\to W_i$ be Mori conic bundles that are geometrically rational with birational maps $\Phi_i\colon X_i\dashto Y_{i}$ of Galois depth $\geq\Delta$ and birational maps (for $i\neq 0$) $\Psi_i\colon Y_{i-1}\dashto X_{i}$ of Galois depth $<\Delta$ such that the diagram
  \begin{center}
    \begin{tikzpicture}
      \matrix (m) [matrix of math nodes,row sep=1em,column sep=2em,minimum width=2em]
      {
      ~ & X_0 & Y_0 & X_1 & & Y_{N-1} & X_N & Y_N & ~\\
      & V_0 & W_0 & V_1 & & W_{N-1} & V_N & W_N & \\};
      \path[-stealth]
      (m-1-1) edge[white] node[above] {~} (m-1-2)
      (m-1-2) edge[dashed] node[above] {$\Phi_0$} (m-1-3)
      (m-1-3) edge[dashed] node[above] {$\Psi_1$} (m-1-4)
      (m-1-4) edge[dashed] node[above] {} (m-1-6)
      (m-1-6) edge[dashed] node[above] {$\Psi_{N}$} (m-1-7)
      (m-1-7) edge[dashed] node[above] {$\Phi_N$} (m-1-8)
      (m-1-8) edge[white] node[above] {~} (m-1-9)
      (m-1-2) edge node {} (m-2-2)
      (m-2-2) edge[white] node[black] {$\simeq$} (m-2-3)
      (m-1-3) edge node {} (m-2-3)
      (m-2-3) edge[] node[black] {$\not$} (m-2-4)
      (m-1-4) edge node {} (m-2-4)
      (m-1-6) edge node {} (m-2-6)
      (m-2-6) edge[] node[black] {$\not$} (m-2-7)
      (m-1-7) edge node {} (m-2-7)
      (m-2-7) edge[white] node[black] {$\simeq$} (m-2-8)
      (m-1-8) edge node {} (m-2-8)
      ;
      \end{tikzpicture}
    \end{center}
    commutes, where we mean by ``~$W_{i-1} \not\to V_i$'' that there is no such morphism making the diagram commute.

 Let $\varphi=\Phi_N\Psi_{N}\Phi_{N-1}\cdots\Phi_1\Psi_1\Phi_0$, and let $H\sim -\lambda K_X+\nu f$ be a linear system without fixed component on $X=X_0$ and let $H'\sim -\lambda' K_{X'}+\nu' f$ be its strict transform in $X'=Y_N$ under $\varphi$.
  Then $\lambda'>\lambda$.
  In particular, there is no morphism $V_0\to W_N$ making the diagram
  \begin{center} \begin{tikzpicture}
   \matrix (m) [matrix of math nodes,row sep=1em,column sep=1em,minimum width=2em]
   {
   & & X_{0} & & Y_{N} \\
   & &V_0 & & W_N \\};
   \path[-stealth]
   (m-1-3) edge[dashed] node[above] {$\varphi$} (m-1-5)
   (m-1-3) edge node {} (m-2-3)
   (m-2-3) edge[] node[black] {$\not$} (m-2-5)
   (m-1-5) edge node {} (m-2-5)
   ;
 \end{tikzpicture}
 \end{center} commute, hence $\varphi$ is not an isomorphism of Mori conic bundles.
\end{corollary}

\begin{proof}

  This is a direct corollary from Lemma~\ref{lemma:InductionStepLambda}:
  We can assume that $H$ is smooth, hence $\mu=0$, and we can apply the lemma.

  For the last part: If $\varphi$ would preserve the fibration, then it would be of type II, hence we would have $\lambda'=\lambda$, a contradiction to Lemma~\ref{lemma:InverseSarkisovProgram}.
\end{proof}

\subsection{Generating relations}\label{section:GeneratingRelations}

\begin{proof}[Proof of Theorem~\ref{theorem:GeneratingRelations}]
  The statement was already proven in Lemma~\ref{lemma:NotGeometricallyRational} if $X$ is not geometrically rational.
  So we assume now that $X$ is geometrically rational.

  Let $\varphi_n\cdots\varphi_1=\id$ be a relation in $\BirMori(X)$, where $\varphi_i\colon Z_{i-1}\dashto Z_i$ is a Sarkisov link of Galois depth $d_i$.
  If all $d_i\leq 15$, we are in situation~\ref{item:GeneratingRelations--Small}.

  Otherwise, the base points of at least one of the $\varphi_i$ contains an orbit of size $\geq 16$.
  In particular, $\varphi_i\colon Z_{i-1}\dashto Z_i$ is a link of type II between Mori conic bundles (since these are the only links of big Galois depth, see Remark~\ref{remark:SmallGalois depth}).
  We will prove that we are always in the situation of Lemma~\ref{lemma:RelationsGeneratedByFour} using Corollary~\ref{corollary:NotAnIsoOfConicBundles}.
  By replacing the relation with
  \[
  \varphi_{i-1}\cdots\varphi_1\varphi_n\cdots\varphi_{i+1}\varphi_i,
  \]
  we can assume that $Z_0$ is a Mori conic bundle.
  We consider the relator $\varphi=\varphi_n\cdots\varphi_1$ and write it -- as in Remark~\ref{remark:DecompositionOfBirMaps} -- as
  \[
  \varphi=\Phi_N \Psi_{N}\cdots \Phi_1\Psi_1\Phi_0 ,
  \]
  where for $i=0,\ldots, N$ the $X_i\to V_i$ and $Y_i\to W_i$ are Mori conic bundles with birational maps $\Phi_i\colon X_i\dashto Y_{i}$ that are a composition of links of type II between Mori conic bundles, and birational maps (for $i\neq 0$) $\Psi_i\colon Y_{i-1}\dashto X_{i}$ of Galois depth $\leq 15$.

  If $N=0$ then $\varphi=\Phi_0$ is a composition of links of type II between Mori conic bundles. The result follows with Lemma~\ref{lemma:RelationsGeneratedByFour}.

  If $N\geq1$, we can assume with Corollary~\ref{corollary:MoveSmallGalois depthToEnd} that each $\Phi_i$ is either a product of links of Galois depth $\geq 16$ or an isomorphism.
  Now, we change our decomposition of the relator $\varphi$ such that it is of the form of Corollary~\ref{corollary:NotAnIsoOfConicBundles}.
  If one of the $\Phi_i$ is an isomorphism, we look at the birational map $\Psi'_i=\Psi_{i+1}\Phi_i\Psi_i:Y_{i-1}\dashto X_{i+1}$.
  There are two possibilities:
  Either $\Psi'_i$ preserves the fibration, or it does not.
  If it does not, we replace $\Psi_{i+1}\Phi_i\Psi_i$ with $\Psi_i'$ in the decomposition of $\varphi$. Note that $\Psi_i'$ is of Galois depth $\leq 15$.
  If $\Psi_i'$ preserves the fibration, we replace $\Phi_{i-1}$ with $\Phi_{i-1}'=\Phi_{i+1}\Psi_i'\Phi_{i-1}\colon X_{i-1}\dashto Y_{i+1}$.
  Applying Corollary~\ref{corollary:MoveSmallGalois depthToEnd} once more, we can assume that $\Phi_{i-1}'$ is a product of links of Galois depth $\geq 16$ or an isomorphism. In the latter case we repeat the process.

  In this way, we arrive either at the case $N=0$ and we are done, or we are in the situation of Corollary~\ref{corollary:NotAnIsoOfConicBundles}, which implies that the relator $\varphi$ is not an isomorphism, a contradiction.
\end{proof}

\section{Detour to Galois theory for non-experts}\label{section:GaloisTheory}

In this section we will prove that a perfect field $\kk$ with $[\overbar\kk:\kk]>2$ contains an arbitrarily large Galois orbit.
Moreover, for any integer $\Delta\geq 2$ (for example $\Delta=8$) we construct in Lemma~\ref{lemma:FieldWithouthSomeOrbit} a perfect field with $[\bar\kk:\kk]>2$ that has no Galois orbit of size exactly $\Delta$, implying that our approach works for slightly more fields than the approach via Bertini involutions of \cite{LZ19}.
We recall first the statements from Galois theory that we need.

A field $\kk$ is called \textit{perfect} if every algebraic extension is separable.
In particular, any finite extension of a perfect field is again perfect.
A finite field extension $L/K$ is called \textit{Galois} if it is normal and separable.
In this case, the extension degree $[L:K]$ equals the number of elements in the Galois group $\Gal(L/K)$.

Moreover, for any splitting field $L$ of an irreducible polynomial $f\in\kk[x]$, the extension $L/\kk$ is normal.
So if $\kk$ is perfect, then $L/\kk$ is Galois.
We will use the Artin-Schreier Theorem \cite[{Corollary 9.3, Chapter IV}]{lang05} and the Primitive Element Theorem \cite[Theorem 4.6, Chapter V]{lang05}.

\begin{lemma}  \label{theorem:DegreeOfExtensionEqualssizeOfOrbit}
  Let $L/\kk$ be a finite Galois extension.
  For $\gamma\in L$ the degree of $[\kk(\gamma):\kk]$ equals the length of the orbit of $\gamma$ under the action of $\Gal(L/\kk)$.
\end{lemma}

\begin{proof}
  As $L/\kk$ is normal and separable by assumption, also $L/\kk(\gamma)$ is normal \cite[Chapter V, Theorem 3.4]{lang05} and separable, hence it is Galois.
  Note that the stabilizer of $\gamma$ is $\Gal(L/\kk(\gamma))$.
  By the orbit formula, the length of the orbit of $\gamma$ under the action of $\Gal(L/\kk)$ equals $[\Gal(L/\kk):\Gal(L/\kk(\gamma))]$, which is equal to \[\frac{\abs{\Gal(L/\kk)}}{\abs{\Gal(L/\kk(\gamma))}}=\frac{[L:\kk]}{[L:\kk(\gamma)]}=[\kk(\gamma):\kk].\]
\end{proof}

\begin{lemma}\label{lemma:LargeGaloisOrbitExists}
  Let $\kk$ be a perfect field with $[\overbar\kk:\kk]>2$ and let $\Delta\geq1$.
  Then $\overline\kk$ contains an orbit of length $\geq \Delta$ under the action of $\Gal(\overbar\kk/\kk)$.
\end{lemma}

\begin{proof}
  The Artin-Schreier Theorem directly implies that the degree $[\overbar\kk:\kk]$ is infinite, and hence for any finite field extension $L/\kk$ with $L\subset \overbar\kk$ we have that $L$ is not equal to the algebraic closure $\overbar\kk$.
  We inductively construct a series of finite field extensions $L_n/\kk$ such that $[L_n:\kk]\geq2^{n}$.
  For the base case, set $L_0=\kk$ so $[L_0:\kk]=1=2^0$ is finite.
  For the induction step $n-1\to n$, assume that there is a finite field extension $L_{n-1}/\kk$ with $[L_{n-1}:\kk]\geq 2^{n-1}$.
  Hence, $L_{n-1}\neq\overbar\kk$ and so there exists $\alpha_n\in\overbar\kk\setminus L_{n-1}$.
  Set $L_n=L_{n-1}(\alpha_n)$, so $L_n/L_{n-1}$ is a finite extension. As $L_n\neq L_{n-1}$ we have $[L_n:L_{n-1}]\geq 2$.
  The induction hypothesis implies with $[L_n:\kk]=[L_n:L_{n-1}][L_{n-1}:\kk]$ that
  \[
    \infty>[L_n:\kk]\geq 2\cdot2^{n-1}=2^n,
  \]
  which means that $L_n/\kk$ is a finite extension of degree $\geq 2^n$.

  Now, choose $n$ such that $2^n\geq \Delta$.
  As $L_n/\kk$ is an algebraic extension of the perfect field $\kk$, it is a separable extension.
  The Primitive Element Theorem can be applied and provides the existence of  $\gamma\in L_n$ such that $L_n=\kk(\gamma)$.
  Take a finite Galois extension $L/\kk$ with $\gamma\in L\subset\overbar\kk$. (For example take $L$ to be the splitting field in $\overbar\kk$ of the minimal polynomial of $\gamma$ over $\kk$.)
  By Lemma~\ref{theorem:DegreeOfExtensionEqualssizeOfOrbit}, the orbit of $\gamma$ is of length $[\kk(\gamma):\kk]=[L_n:\kk]\geq \Delta$.
\end{proof}

Note that we do not claim that an orbit of exact size $\Delta$ exists.
In fact, for any $\Delta\geq 2$ there exists a perfect field with no Galois orbit of exact size $\Delta$, namely the following example that was provided to me by Lars Kuehne.

\begin{lemma}\label{lemma:FieldWithouthSomeOrbit}
  Let $\Delta\geq 2$.
  There exists a perfect field $\kk$ with $[\overbar\kk:\kk]=\infty$ such that no element in $\overbar\kk$ has an orbit of length $\Delta$ under the action of $\Gal(\overbar\kk/\kk)$.
\end{lemma}

\begin{proof}
  Consider the field extension $\QQQ\subset \kk\subset\overbar\QQQ$, where $\kk$ is the set consisting of elements $a\in\overbar\QQQ$ such that there exists a tower of fields $\QQQ=L_0\subset L_1\subset\cdots\subset L_n\ni a$ such that $L_i/L_{i-1}$ is the splitting field of a polynomial of degree $\Delta$ with coefficients in $L_{i-1}$.

  Indeed, $\kk$ is a (perfect) field: For $a,b\in \kk$ let $L_0\subset L_1\subset \cdots\subset L_n\ni a$ be the tower of fields corresponding to $a$, and let $g_i$ be the polynomials of degree $\Delta$ corresponding to the splitting fields corresponding to $b$ for $i=1,\ldots, m$.
  There exists a tower of fields $L_0\subset\cdots \subset L_{n}\subset L_{n+1}\subset \cdots L_{n+m}$, where $L_{n+i}/L_{n+i-1}$ is the splitting field of $g_i$ (seen as a polynomial with coefficients in $L_{n+i-1}$).
  Since $L_{n+m}$ is a field containing $a$ and $b$, it also contains $a+b$ and $ab$.
  So alos $\kk$ contains $a+b$ and $ab$.
  Therefore, $\kk$ is a field, and it is perfect because its characteristic is zero.

  To prove that $[\overbar\QQQ:\kk]=\infty$, we assume that $[\overbar\QQQ:\kk]=N$ for some $N\in\NNN$.
  Let $p>\max\{N,\Delta!\}$ be a prime number.
  First, we prove that there exists a Galois extension $F/\QQQ$ of degree $p$.
  By Dirichlet's Theorem, one can choose a prime $q\equiv 1\mod p$.
  Let $\QQQ(\mu_q)$ be the cyclotomic extension of $\QQQ$, where $\mu_q$ is a $q^{\mathrm{th}}$ root of unity.
  The Galois group of $\QQQ(\mu_q)/\QQQ$ is the multiplicative group $(\ZZZ/q\ZZZ)^{\times}$, which is cyclic of order $q-1$.
  As $p$ divides $q-1$, there exists a (normal) subgroup $H\subset \Gal(\QQQ(\mu_q)/\QQQ)$ of order $\frac{q-1}{p}$.
  Let $F\subset\QQQ(\mu_q)$ be the field that is fixed by $H$.
  By Galois Theory, the extension $F/\QQQ$ is Galois and of degree $p$ (using that the extension degree of  $\QQQ(\mu_q)/\QQQ$ is $q-1$) \cite[Chapter VI, Theorem 1.1 and 1.8]{lang05}.

  By the Primitive Element Theorem, we can choose $\alpha\in\overbar\QQQ$ such that
  $F= \QQQ(\alpha)$.
  Hence $\QQQ(\alpha)/\QQQ$ is a Galois extension with $[\QQQ(\alpha):\QQQ]=p$.
  We prove that $\alpha\in\kk$.
  As $\kk/\QQQ$ is an (arbitrary) extension, the degree $[\kk(\alpha):\kk]$ divides $[\QQQ(\alpha):\QQQ]=p$ \cite[{Chapter VI, Corollary 1.13}]{lang05} (using that the compositum $\QQQ(\alpha)\kk$ of the two fields $\QQQ(\alpha)$ and $\kk$ equals $\kk(\alpha)$).
  By the transitivity of the degree, we also find that
  \[
    N=[\overbar\QQQ:\kk]=[\overbar\QQQ:\kk(\alpha)][\kk(\alpha):\kk],
  \]
  so $[\kk(\alpha):\kk]$ divides $N$.
  Since it also divides the prime number $p>N$, the only possibility is $[\kk(\alpha):\kk]=1$ and so $\alpha\in\kk$.

  Now, we find a contradiction to $p>\Delta!$.
  As $\alpha$ lies in $\kk$, there exists a tower of fields $L_0\subset\cdots \subset L_n\ni \alpha$ such that $L_i/L_{i-1}$ is the splitting field of a polynomial of degree~$\Delta$.
  Hence, $[L_i:L_{i-1}]\leq \Delta!$.
  So $[L_n:\QQQ]=[L_n:L_{n-1}]\cdots[L_1:\QQQ]$ is a product of numbers smaller than or equal to $\Delta!$.
  Note that $\QQQ(\alpha)\subset L_n$, hence $[L_n:\QQQ]=[L_n:\QQQ(\alpha)][\QQQ(\alpha):\QQQ]$ and so $p=[\QQQ(\alpha):\QQQ]$ divides $[L_n:\QQQ]$.
  As $p$ is a prime, it implies that $p\leq\Delta!$, which is a contradiction to $p>\Delta!$.

  Finally, we prove that $\kk$ has no Galois orbit of size $\Delta$.
  Assume that there exists $\beta\in\overbar\QQQ$ such that $\kk(\beta)/\kk$ is finite and such that the length of the Galois orbit of $\beta$ is $\Delta$.
  In particular, $\beta\in\overbar\QQQ\setminus\kk$.
  Consider the minimal polynomial $\min_\kk(\beta)$ of $\beta$ over $\kk$ and its splitting field $L$.
  So $\beta\in L$ and $L/\kk$ is finite and Galois.
  Hence with Theorem~\ref{theorem:DegreeOfExtensionEqualssizeOfOrbit} we have that the size of the Galois orbit of $\beta$, which is $\Delta$, equals $[\kk(\beta):\kk]$, which in turn is the degree of the minimal polynomial $\min_\kk(\beta)$.
  By the construction of our field $\kk$, this implies that $\beta$ already lies in $\kk$, a contradiction.
\end{proof}

\section{Group homomorphism}\label{section:GroupHomo}

\begin{definition}\label{def:EquivalentConicBundles}
  We say that two Mori conic bundles $X/W$ and $X'/W'$ are \textit{equivalent} if there exists a birational map $X\dashto X'$ that preserves the fibration (see Definition~\ref{definition:PreserveFibrationAndIsoOfMFS}).
\end{definition}

We denote the set of equivalence classes of Mori conic bundles birational to $X$ by $\CB(X)$.

\begin{definition}\label{def:EquivalentSarkisovLinks}
  We say that two Sarkisov links $\chi$ and $\chi'$ of Mori conic bundles of type II are \textit{equivalent}, if \begin{enumerate}
    \item the Mori conic bundles are equivalent,
    \item the Sarkisov links have the same Galois depth.
  \end{enumerate}
\end{definition}

For an equivalence class $C\in\CB(X)$ of Mori conic bundles, we denote by $\Marked(C)$ the set of equivalence classes of Sarkisov links of type II (between Mori conic bundles in the equivalence class of $C$).
That is, an element of $\Marked(C)$ is the class of Sarkisov links of type II between equivalent Mori conic bundles of the same Galois depth.

\begin{proof}[Proof of Theorem~\ref{theorem:GroupHomomorphism}]
  One has to show that the homomorphism is well defined, that is, to show that every relator is mapped onto the identity.

  By construction, relators that consist of Sarkisov links of Galois depth $\leq 15$ are mapped on the identity.

  The trivial relation $\alpha\beta=\gamma$, where $\alpha,\beta,\gamma$ are isomorphisms of Mori fiber spaces, is mapped onto the identity by construction.
  Trivial relations of the form $\alpha\psi^{-1}\phi=\id$ satisfy $\Bas(\phi)=\Bas(\psi)$, hence they have the same Galois depth and are therefore in the same equivalence class of Sarkisov links.
  Hence $\psi$ and $\phi$ have the same image and so the relator is mapped onto the identity.

  Relations of the form $\chi_4\chi_3\chi_2\chi_1=\id$, where $\chi_i$ are Sarkisov links of type II between Mori conic bundles, are such that $\chi_4$ and $\chi_2$, as well as $\chi_3$ and $\chi_1$ have the same Galois depth as in Remark~\ref{remark:RelationOfFourLinks}.
  As they are all links between Mori conic bundles of the same equivalence class of Mori conic bundles, $\chi_1$ and $\chi_3$, as well as $\chi_2$ and $\chi_4$ have the same image.
  Therefore, the relator is mapped onto the identity.
  This proves the existence of the groupoid homomorphism.
  The fact that it restricts to a group homomorphism from $\Bir(X)$ is immediate, and the fact that it restricts to a group homomorphism from $\Bir(X/W)$ is a consequence of Corollary~\ref{corollary:PreservingTheFibrationMeansTypeII}.
\end{proof}

Having established Theorem~\ref{theorem:GroupHomomorphism}, it is now not hard to prove Theorem~\ref{theorem:CremonaNotSimple}:
We take the restriction of the group homomorphism to the equivalence class of $\PPP^1\times\PPP^1$ (the Hirzebruch surfaces), and construct links of type II of large Galois depth.

\begin{example}\label{example:P1P1NonTrivialImage}
  Consider the birational map
  \[
    (x,y)\mapsto (xp(y),y),
  \]
  and its extension to a birational map  $\varphi\colon\PPP^1\times\PPP^1\dashto\PPP^1\times\PPP^1$ that is given by
  \[
    [x_0:x_1;y_0:y_1]\mapsto [x_0y_1^d:x_1p(y_0,y_1); y_0:y_1],
  \]
  where $p\in\kk[y_0,y_1]$ is an irreducible polynomial of degree $d\geq 16$.
  Since $\kk$ is perfect, $p(t,1)$ has $d$ different zeroes $t_1,\ldots,t_d\in\overbar\kk$.
  So $\varphi$ is not defined on $p=[1:0;1:0]$ and on the points $p_i=[0:1;t_i:1]$ for $i=1,\ldots,d$.
  We check that $\varphi$ is the composition of a link $\varphi_0\colon\PPP^1\times\PPP^1\dashto \FFF_d$ of type II centered at the orbit $\{p_1,\ldots,p_d\}$, followed by $d$ links $\varphi_n\colon \FFF_n\dashto\FFF_{n-1}$ of type II of Galois depth $1$ for $n=d,\ldots,1$:

  To resolve $\varphi$ we need to resolve the linear system of bidegree $(1,d)$ given by $[x_0y_1^d:x_1p(y_0,y_1)]$.
  It consists of smooth curves and has self-intersection $2d$, so we need to blow up $2d$ points with multiplicity $1$.
  With $\varphi_0$ and a link $\varphi_d\colon \FFF_d\dashto\FFF_{d-1}$ of type II centered at the image of $p$ in $\FFF_d$, we have resolved the $d+1$ base points that we see from the equation.
  Hence, the remaining $d-1$ base points are infinitely near to them.
  As the exceptional divisors of the $p_i$ form an orbit of size $d$, the $d-1$ base points are infinitely near to $p$.
  Note that $p$ lies on the section $s$ given by $x_1=0$, which is mapped onto itself by $\varphi$, and that the $p_i$ do not lie on $s$.
  The image of $s$ in $\FFF_{d-1}$ has self-intersection $d-1$ and so the remaining $d-1$ base points lie on the strict transform of $s$, each of them giving a link $\varphi_n\colon \FFF_n\dashto\FFF_{n-1}$ of type II of Galois depth $1$, for $n=d-1,\ldots,1$.

  Note that $\varphi_0$ is not mapped onto the identity (its image is $1\in\ZZZ/2\ZZZ$ corresponding to the equivalence class of $\varphi_0$ in $\Marked(\PPP^1\times\PPP^1)$), whereas all $\varphi_n$ for $n=1,\ldots,d$ are mapped onto the identity.
  Therefore, the image of $\varphi\in\Bir(\PPP^1\times\PPP^1)$ under the group homomorphism is non-trivial.
\end{example}

\begin{proof}[Proof of Theorem~{\ref{theorem:CremonaNotSimple}}]
  We take the group homomorphism from Theorem~\ref{theorem:GroupHomomorphism}.
  For a constant polynomial $p\in \kk$, the local map $(x,y)\mapsto(px,y)$ is an automorphism and therefore it is mapped onto the identity.
  For the surjectivity, using Lemma~\ref{lemma:LargeGaloisOrbitExists} we can construct an infinite and countable indexing set $I$ such that for each $d\in I$ there exists an irreducible polynomial $p\in\kk[y]$ of degree $d$, and each $d\in I$ is at least $16$.
  For each such polynomial we consider $\varphi\colon\PPP^1\times\PPP^1\to\PPP^1\times\PPP^1$ as in  Example~\ref{example:P1P1NonTrivialImage}.
  Let $\alpha\colon\PPP^2\dashto\PPP^1\times\PPP^1$ be the blow-up of two points in $\PPP^2$ that are defined over $\kk$, followed by the contraction of the line connecting the two points.
  Then $\alpha^{-1}\varphi\alpha$ lies in $\Bir(\PPP^2)$ and, since $\alpha$ and $\alpha^{-1}$ have Galois depth $1$, its image under the group homomorphism of Theorem~\ref{theorem:GroupHomomorphism} is non-trivial on the index of $I$ corresponding to the degree of $p$.
\end{proof}

\section{Rational Mori conic bundles}\label{section:RationalMoriCB}

In this section, we are interested in rational surfaces -- especially (Mori) conic bundles -- over an arbitrary perfect field $\kk$.
In Proposition~\ref{proposition:SameFourPoints} we find a way to distinguish equivalence classes of Mori conic bundles, and in Lemma~\ref{lemma:AllMoriConicBundles} we find that rational Mori conic bundles are either a Hirzebruch surface, or given by the blow-up of an orbit of size $4$ in $\PPP^2$, or given by the blow-up of two orbits of size $2$ in $\PPP^2$, followed by the contraction of the strict transform of the line between two of them.
Building upon that we remark in Corollary~\ref{corollary:GeneratingSets} that $\Bir_\kk(\PPP^2)$ is generated by the Jonqui\`eres maps, the sets of birational maps that preserve the pencil of conics through an orbit of size $4$ or two orbits of size $2$, and the set of birational maps with Galois depth at most $8$.
Finally, we prove Theorem~\ref{theorem:Refinement} (the refinement of Theorem~\ref{theorem:CremonaNotSimple}) in Section~\ref{section:freeproduct}, and give a visualization of the long list of Sarkisov links of \cite[Theorem 2.6]{Iskovskikh96} in Section~\ref{section:longlist}.

\subsection{Geography}\label{section:geography}

\begin{remark}\label{remark:geography}
  The list of Sarkisov links in \cite[Theorem 2.6]{Iskovskikh96} -- see Section~\ref{section:longlist} for a quick overview -- implies that a rational Mori conic bundle $X\to\PPP^1$ satisfies either $K_X^2=8$ and $X$ is a Hirzebruch surface, or $K_X^2\in\{5,6\}$.
\end{remark}

To study the rational Mori conic bundles with $K_X^2\in\{5,6\}$ simultaneously, we are going to take a closer look at the blow-up of $\PPP^2$ at four points.

\begin{remark}\label{remark:BlowUpOf4Pts}
  Let $p_1,p_2,p_3, p_4\in\PPP^2(\bar\kk)$ be four points, no three of them collinear, and let $\pi\colon X\to\PPP^2$ be the blow-up centered at the four points (defined over $\bar\kk$).
  There is a morphism $X\to\PPP^1$ whose fibers are the strict-transforms of the conics of $\PPP^2$ passing through the four points.
  The morphism has three singular fibers, namely the strict transforms of conics consisting of two lines each through two of the points.

  Assuming that the set $\{p_1,\ldots,p_4\}$ is invariant under the action of $\Gal(\bar\kk/\kk)$, then the morphisms $\pi\colon X\to\PPP^2$ and $X\to\PPP^1$ are defined over $\kk$, and the union of the three singular fibers is invariant under $\Gal(\bar\kk/\kk)$.
\end{remark}

\begin{definition}
  Let $X\to\PPP^1$ be a conic bundle.
  We call a section (defined over $\bar\kk$) that is a $(-1)$-curve a \textit{$(-1)$-section}.
  A $(-1)$-curve that is an irreducible component of a singular fiber will be called a \textit{vertical $(-1)$-curve}.
\end{definition}

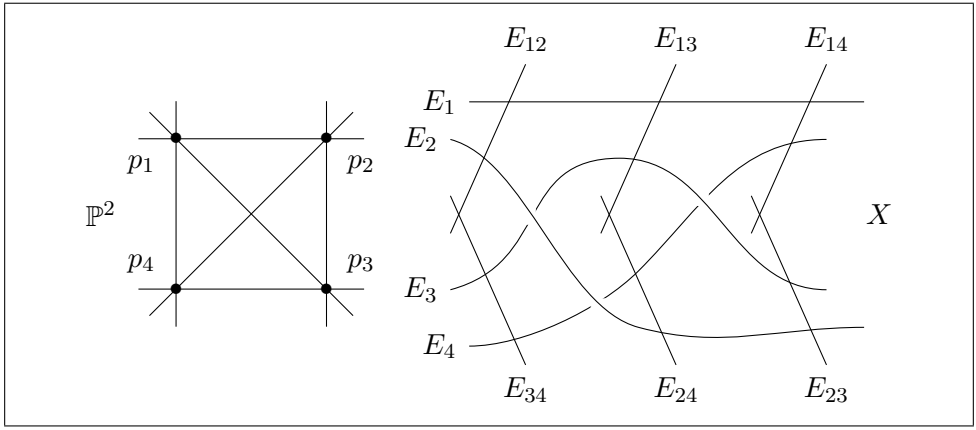
\begin{figure}\label{figure:tenminusonecurves}
  \caption{The situation of {Lemma~\ref{lemma:BlowUpOfFourPointsDescription}}.
  On the left: The three singular conics through $p_1$, $p_2$, $p_3$ and $p_4$ of~\ref{item:BlowUpOfFourPointsDescription--Blowup}.
  On the right: The ten $(-1)$-curves as in \ref{item:BlowUpOfFourPointsDescription--ConfigurationOfX}.}
  \begin{center}
    \begin{tikzpicture}[]
      \node [label={[label distance=-0.1cm]210:$p_1$}] (0) at (-1, 1) {$\bullet$};
  		\node [label={[label distance=-0.1cm]-30:$p_2$}] (1) at (1, 1) {$\bullet$};
  		\node [label={[label distance=-0.1cm]30:$p_3$}] (2) at (1, -1) {$\bullet$};
  		\node [label={[label distance=-0.1cm]150:$p_4$}] (3) at (-1, -1) {$\bullet$};
      \node [] (base) at (0, -2.5) {};
      \node () at (-2,0) {$\PPP^2$};
      \draw [shorten >=-0.5cm, shorten <= -0.5cm] (0.center) to (3.center);
  		\draw [shorten >=-0.5cm, shorten <= -0.5cm] (2.center) to (1.center);
  		\draw [shorten >=-0.5cm, shorten <= -0.5cm] (0.center) to (2.center);
  		\draw [shorten >=-0.5cm, shorten <= -0.5cm] (1.center) to (3.center);
  		\draw [shorten >=-0.5cm, shorten <= -0.5cm] (2.center) to (3.center);
  		\draw [shorten >=-0.5cm, shorten <= -0.5cm] (0.center) to (1.center);
    \end{tikzpicture}%
    ~~~~~~
    \begin{tikzpicture}
      \node () at (2.7,0) {$X$};
      \node [label={[label distance=-0.1cm]90:$E_{12}$}] (0) at (-2, 2) {};
  		\node [label={[label distance=-0.1cm]90:$E_{13}$}] (1) at (0, 2) {};
  		\node [label={[label distance=-0.1cm]90:$E_{14}$}] (2) at (2, 2) {};
  		\node [] (4) at (-3, -0.25) {};
  		\node  (5) at (-1, -0.25) {};
  		\node  (6) at (1, -0.25) {};
  		\node  (8) at (-3, 0.25) {};
  		\node (9) at (-1, 0.25) {};
  		\node  (10) at (1, 0.25) {};
  		\node [label={[label distance=-0.1cm]270:$E_{34}$}] (12) at (-2, -2) {};
  		\node [label={[label distance=-0.1cm]270:$E_{24}$}] (13) at (0, -2) {};
  		\node [label={[label distance=-0.1cm]270:$E_{23}$}] (14) at (2, -2) {};
  		\node [label={[label distance=-0.1cm]180:$E_1$}] (15) at (-2.75, 1.5) {}; 
  		\node [label={[label distance=-0.1cm]180:$E_2$}] (16) at (-3, 1) {};
  		\node [label={[label distance=-0.1cm]180:$E_3$}](17) at (-3, -1) {};
  		\node [label={[label distance=-0.1cm]180:$E_4$}](18) at (-2.75, -1.75) {};
  		\node (19) at (2.5, 1.5) {};
  		\node (20) at (2, 1) {};
  		\node (21) at (2.5, -1.5) {};
  		\node (22) at (2, -1) {};
  		\node (24) at (-0.75, 0.75) {};
  		\node (25) at (-0.5, -1.5) {};
  		\node (26) at (-0.5, -0.75) {};

      \draw (18.center)
  			 to [in=-135, out=0, looseness=0.75] (26.center)
  			 to [in=-180, out=45] (20.center);
  		\draw (15.center) to (19.center);
      \node [fill=white, shape=rectangle,draw=white] (29) at (0.3, 0.25) {};

      \draw (17.center)
  			 to [in=180, out=15, looseness=1.25] (24.center)
  			 to [in=-180, out=0] (22.center); 

      \node [fill=white, shape=rectangle,draw=white] (27) at (-2, 0) {};
      \node [fill=white, shape=rectangle,draw=white] (28) at (-1, -1.25) {};

  		\draw (16.center)
  			 to [in=165, out=-15, looseness=0.75] (25.center)
  			 to [in=-180, out=-15] (21.center);
  		\draw (0.center) to (4.center);
  		\draw (8.center) to (12.center);
  		\draw (1.center) to (5.center);
  		\draw (9.center) to (13.center);
  		\draw (10.center) to (14.center);
  		\draw (6.center) to (2.center);
    \end{tikzpicture}
  \end{center}
\end{figure}

\begin{lemma}\label{lemma:BlowUpOfFourPointsDescription}
  Let $X\to\PPP^1$ be a conic bundle with a $\kk$-point in $X$.
  Then the following are equivalent:
  \begin{enumerate}
    \item\label{item:BlowUpOfFourPointsDescription--Blowup} There exists a morphism $X\to\PPP^2$ that is the blow-up of four points $p_1,\ldots,p_4\in\PPP^2(\kk)$, no three collinear, and the set of the four points is invariant under the action of $\Gal(\bar\kk/\kk)$.
    \item\label{item:BlowUpOfFourPointsDescription--ConfigurationOfX} $X$ satisfies $K_X^2=5$ and it contains ten $(-1)$-curves $E_{ij}$, $1\leq i\leq j\leq 4$, whose union is invariant under the action of $\Gal(\bar\kk/\kk)$ such that (see also Figure~\ref{figure:tenminusonecurves})
    \begin{enumerate}
      \item each $E_i:=E_{ii}$ for $i=1,\ldots,4$ is a $(-1)$-section, and
      \item $E_{12}+E_{34}$, $E_{13}+E_{24}$, and $E_{14}+E_{23}$ are three singular fibers,
      \item\label{item:BlowUpOfFourPointsDescription--Intersection} $E_k$ intersects $E_{ij}$ if and only if $k\in\{i,j\}$.
    \end{enumerate}
  \end{enumerate}
  In this case, $X$ is a del Pezzo surface of degree $5$.
\end{lemma}

\begin{proof}
  Let $X\to\PPP^2$ be the blow-up of the four points  from~\ref{item:BlowUpOfFourPointsDescription--Blowup}.
  Since no three of the points are collinear, $X$ is a del Pezzo surface of degree $5$.
  For $i\neq j$ let $E_{ij}$ be the strict transform of the line through $p_i$ and $p_j$, and let $E_i$ be the exceptional divisor of $p_i$, for $i=1,\ldots,4$.
  As in Remark~\ref{remark:BlowUpOf4Pts}, there is a fibration $X\to\PPP^1$ where the fibers correspond to conics through $p_1,\ldots,p_4$.
  So the only singular fibers are $E_{ij}+E_{kl}$, where $i,j,k,l$ are pairwise distinct. These are exactly three.
  The intersection of the $E_i$ and $E_{ij}$ are as in~\ref{item:BlowUpOfFourPointsDescription--Intersection}.

  For the converse direction, let $X$ be as in~\ref{item:BlowUpOfFourPointsDescription--ConfigurationOfX}.
  As the union of the ten $(-1)$-curves is invariant under the action of $\Gal(\bar\kk/\kk)$, and the morphism $X\to\PPP^1$ is defined over $\kk$, so the Galois action maps fibers onto fibers, and so the union of the other four $(-1)$-curves $E_1,\ldots,E_4$ is again invariant under the Galois action.
  So the contraction $X\to Y$ of $E_1,\ldots,E_4$ onto points $p_1,\ldots,p_4$ is defined over $\kk$.
  As $X$ contains a $\kk$-point and since $K_Y^2=5+4=9$, we have $Y=\PPP^2$.
  The blow-down of the vertical $(-1)$-curves $E_{ij}$ are the lines through the points $p_i$ and $p_j$.
  If three of the four points were collinear, the line through them would have self-intersection $-2$. This is not possible since by assumption, each $E_{ij}$ has self-intersection $-1$.
\end{proof}

Recall that in Remark~\ref{remark:ConicBundleOverP1} we have observed that a conic bundle $X\to\PPP^1$ with $K_X^2=5$ has $3$ singular fibers, and in Remark~\ref{remark:MFSToPointOrCurve} we have seen that if it is moreover a Mori conic bundle, then the two irreducible components of any singular fiber lie in the same Galois orbit, which then implies that there exists $\sigma\in\Gal(\bar\kk/\kk)$ that exchanges the two components.

\begin{lemma}\label{lemma:ThreeSingularFibers}
  Let $X\to\PPP^1$ be a conic bundle with $K_X^2=5$ and a $\kk$-point in $X$.
  If there exists $\sigma_1\in\Gal(\bar \kk/\kk)$ that exchanges the two irreducible components of one of the three singular fibers, then the following hold:
  \begin{enumerate}
    \item\label{item:ThreeSingularFibers--MinimalSelfIntersectionSection} The minimal self-intersection of a section on $X$ is $-1$.
    \item\label{item:ThreeSingularFibers--DelPezzoAndBlowUpOf4} There exists a morphism $X\to\PPP^2$ that is the blow-up at four points, no three collinear, and the set of the four points is invariant under the action of $\Gal(\bar\kk/\kk)$.
    In particular, $X$ is a del Pezzo surface of degree $5$.
    \item\label{item:ThreeSingularFibers--ThereAreFourSections} There are exactly four $(-1)$-sections $E_1,\ldots, E_4$ on $X$.
    Moreover, for $i\neq j$ there is exactly one vertical $(-1)$-curve that intersects $E_i$ and $E_j$, and each vertical $(-1)$-curve meets exactly two of the $E_i$.
    \item\label{item:ThreeSingularFibers--Possibilities} Denote by $E_{12}+E_{34}$, $E_{13}+E_{24}$, $E_{14}+E_{23}$ the three singular fibers of $X$. Up to reordering $E_1,\ldots,E_4$ from~\ref{item:ThreeSingularFibers--ThereAreFourSections}, there are only two possibilities:
    \begin{enumerate}
      \item\label{item:ThreeSingularFibers--Confi}
      $E_i$ intersects $E_{jk}$ if and only if $i\in\{j,k\}$, or
      \item\label{item:ThreeSingularFibers--Confii}
      $E_i$ intersects $E_{jk}$ if and only if $i\notin\{j,k\}$.
    \end{enumerate}
  \end{enumerate}
\end{lemma}

\begin{proof}
  By \cite[Lemma 3.3]{Blanc09}, the minimal self-intersection of a section is $-n$ for some $n\geq 1$.
  To show that $n=1$, we assume by contradiction that there is a section $E_1$ with self-intersection $-n$ and $n\geq2$.
  As $E_1$ is a section, its intersection with each fiber is $1$, so also with the singular fiber whose irreducible components are exchanged by $\sigma_1$.
  Hence, $E_1$ meets exactly one of the irreducible components.
  So $E_2=\sigma_1(E_1)$ intersects the other irreducible component, and it is again a section with self-intersection $-n$, implying $E_2\neq E_1$.
  Hence there are two distinct sections with self-intersection $-n$.
  \cite[Lemma 3.3]{Blanc09} implies that the number of singular fibers is at least $2n$, giving $3\geq2n\geq4$, a contradiction.
  We have proven \ref{item:ThreeSingularFibers--MinimalSelfIntersectionSection}.

  We now prove~\ref{item:ThreeSingularFibers--DelPezzoAndBlowUpOf4} and~\ref{item:ThreeSingularFibers--ThereAreFourSections}.
  Let $E_1$ be a section with self-intersection $-1$, which exists by~\ref{item:ThreeSingularFibers--MinimalSelfIntersectionSection}.
  Up to exchanging the names of two irreducible components in the same singular fiber, we can assume that $E_1$ meets $E_{12}$, $E_{13}$, and $E_{14}$.
  Let $X\to Y$ be the contraction (only defined over $\bar\kk$) of the $(-1)$-curves $E_{34}$, $E_{24}$, $E_{23}$ and $E_1$ onto points $F$, $G$, and $H$ and $P$.
  As $X$ contains a $\kk$-point, also $Y$ has one, so $K_Y^2=K_X^2+4=9$ implies $Y=\PPP^2$.
  Note that the image of the fibers are lines through the point $P$.
  So $F$, $G$, and $H$ are not collinear because a line through all of them would be a section of self-intersection $-2$ in $X$, which is not possible by \ref{item:ThreeSingularFibers--MinimalSelfIntersectionSection}.
  So let $E_2=L_{GH}$, $E_3=L_{FH}$, and $E_4=L_{FG}$ be the lines through two of the three points $F$, $G$ and $H$.
  On $X$, these correspond to $(-1)$-sections.
  Note that on $X$, $E_2$ (respectively $E_3$, respectively $E_4$) meets $E_{12}$, $E_{24}$ and $E_{23}$ (respectively $E_{34},E_{13},E_{23}$, respectively $E_{34},E_{24},E_{14}$).
  So $X$ is in the situation of Lemma~\ref{lemma:BlowUpOfFourPointsDescription},\ref{item:BlowUpOfFourPointsDescription--ConfigurationOfX}, and so \ref{item:ThreeSingularFibers--DelPezzoAndBlowUpOf4} holds.
  As $X$ is a del Pezzo surface, there are only ten $(-1)$-curves on $X$ \cite[Section 8.5]{Dolgachev12}, and so $E_1,\ldots,E_4$ are the only $(-1)$-sections on $X$.
  So we see that there is exactly one vertical $(-1)$-curve that intersects $E_i$ and $E_j$.
  If it would intersect a third $E_k$, then there is another vertical $(-1)$-curve that contains two of the $(-1)$-sections, a contradiction.
  This proves~\ref{item:ThreeSingularFibers--ThereAreFourSections}.

  Note that in the proof of ~\ref{item:ThreeSingularFibers--ThereAreFourSections} we exchanged the names of the irreducible components of a singular fiber (if needed).
  For~\ref{item:ThreeSingularFibers--Possibilities}, we fix our fibers and look at the possibilities of how $E_1,\ldots,E_4$ intersect them.
  Looking only at $E_{12}+E_{34}$ and $E_{13}+E_{24}$, there is only one possibility (up to exchanging the order of the $E_i$): $E_1$ intersects $E_{12}$ and $E_{13}$, $E_2$ intersects $E_{12}$ and $E_{24}$, $E_3$ intersects $E_{34}$ and $E_{13}$, and $E_4$ intersects $E_{34}$ and $E_{24}$.
  For the intersection with $E_{14}$ and $E_{23}$ there are now two possibilities:
  \begin{itemize}
    \item Either $E_1$ intersects $E_{14}$ (then $E_2$ and $E_3$ intersect $E_{23}$, $E_4$ intersects $E_{14}$), and we get the first configuration, or
    \item $E_1$ intersects $E_{23}$ (then $E_2$ and $E_3$ intersect $E_{14}$, $E_4$ intersects $E_{23}$).
    Now, we exchange the order of the $E_i$ in the following way: We exchange $E_1$ with $E_4$, and $E_2$ with $E_3$. Then we get the second configuration.
  \end{itemize}
\end{proof}

\begin{observation}\label{observation:TransitiveSubgroupsOfSym4}
  Any transitive subgroup of $\Sym_4$ contains an element that exchanges $\{1,2\}$ with $\{3,4\}$ (respectively $\{1,4\}$ with $\{2,3\}$, respectively $\{1,3\}$ with $\{2,4\}$):
  All transitive subgroups of $\Sym_4$ are
 \begin{enumerate}
    \item $\Sym_4$,
    \item $A_4\subset \Sym_4$,
    \item $D_8=\langle(1234),(13)\rangle\subset A_4$,
    \item $V_4=\{\id,(12)(34),(13)(24),(14)(23)\}\subset D_8$, and
    \item $Z_4=\langle (1234)\rangle=\{\id,(1234),(13)(24),(1432)\}\subset D_8$.
  \end{enumerate}
  An exchange of $\{1,3\}$ with $\{2,4\}$ happens in $V_4$ with the permutation $(14)(23)$, in $Z_4$ there is $(1234)$ doing the same.
  The permutation $(13)(24)$ is contained in all subgroups, and it exchanges $\{1,2\}$ with $\{3,4\}$, and $\{1,4\}$ with $\{2,3\}$.
\end{observation}

\begin{observation}\label{observation:GaloisGroupOfOrbitSize4}
  Let $\PPPP=\{p_1,\ldots, p_4\}\subset\PPP^2(\bar\kk)$ be an orbit of size $4$ of the action of $\Gal(\bar\kk/\kk)$ and let $L_{ij}$ be the line through $p_i$ and $p_j$. As $\Gal(\bar \kk/\kk)$ acts transitively on these four points, $\PPPP$ is defined over a field $L\supset \kk$ with $\Gal(L/\kk)$ a transitive subgroup of $\Sym_4$.
  So up to renumbering the four points $p_1,\ldots,p_4$, the Galois group $\Gal(L/\kk)$ is one of Observation~\ref{observation:TransitiveSubgroupsOfSym4} and it contains a permutation that maps $L_{12}\leftrightarrow L_{34}$ (respectively $L_{14}\leftrightarrow L_{23}$, respectively $L_{13}\leftrightarrow L_{24}$).
\end{observation}

\begin{example}
  Note that $\Gal(L/\kk)$ does not have to be of order $4$:
  Consider the four points $[1:\zeta:\zeta^2]\in\PPP^2(\QQQ)$, where the $\zeta$ are the roots of $x^5-1$ different from $1$.
  Then, $\Gal(\QQQ(\zeta)/\QQQ)$ is the dihedral group $D_8$.
\end{example}

\begin{observation}\label{observation:GaloisGroupOfOrbitSize2}
  Let $p_1,p_3$ and $p_2,p_4$ each be an orbit of size $2$. Then, the action of $\Gal(\bar\kk/\kk)$ on the four points is one of the following: \begin{enumerate}
    \item $\{\id,(13),(24),(13)(24)\}$
    \item $\{\id,(13)(24)\}$.
  \end{enumerate}
  Again, both of these subgroups contain the permutation $(13)(24)$, which maps $L_{12}\leftrightarrow L_{34}$, and $L_{14}\leftrightarrow L_{23}$.
  Moreover, none of these groups contain an element that maps $L_{13}\leftrightarrow L_{24}$.
\end{observation}

\begin{lemma}\label{lemma:IrreducibleComponentsExchanged}
  Let $\varphi\colon X\dashto X'$ be a birational map preserving fibrations $\pi\colon X\to\PPP^1$ and $\pi'\colon X'\to\PPP^1$.
  Assume that each fiber of $\pi'$ contains at most two components.
  Then, for any singular fiber $F=\pi^{-1}(p)$ consisting of two $(-1)$-curves  \textup{(}over $\bar\kk$\textup{)} such that there is $\tau\in\Gal(\bar\kk/\kk)$ that exchanges the irreducible components of $F$, the birational map $\varphi$ is defined on each point of $F$ and it is a local isomorphism at any point of $F$.
\end{lemma}

\begin{proof}
  Let $\rho\colon Z\to X$ and $\rho'\colon Z\to X'$ be the minimal resolution of $\varphi$. Hence, $\rho$ and $\rho'$ are blow-ups of orbits.
  Note that $\rho'$ contracts only curves lying in a fiber, since $\varphi$ preserves the fibrations.
  We prove that none of the irreducible components of $F$ are contracted by $\rho'$:
  If an irreducible component of $F$, say $f_1$, is contracted, then also the other one, say $f_2$, is contracted because $\tau(f_1)=f_2$ is a point.
  But after contracting $f_1$, $f_2$ has self-intersection $0\neq -1$ and can thusly not be contracted.

  We can apply the same reasoning to the inverse of $\varphi$, hence $\varphi$ has no base points and is therefore a local isomorphism on the fiber $F$.
\end{proof}

The following lemma is a slight generalisation of the argument in \cite[Lemma 4.1]{ALNZ19}.

\begin{lemma}\label{lemma:AutoDefinedOverk}
  Let $L/\kk$ be a finite Galois extension \textup{(}that is, normal and separable\textup{)}.
  Let $p_1,\ldots, p_4\in\PPP^2(L)$, no three collinear, and $q_1,\ldots,q_4\in\PPP^2(L)$, no three collinear, be points such that the sets $\{p_1,\ldots,p_4\}$ and $\{q_1,\ldots,q_4\}$ are invariant under the Galois action of $\Gal(L/\kk)$.
  Assume that for all $g\in\Gal(L/\kk)$ there exists $\sigma\in\Sym_4$ such that $g(p_i)=p_{\sigma(i)}$, and $g(q_i)=q_{\sigma(i)}$ for $i=1,\ldots,4$.
  Then, there exists $A\in\PGL_3(\kk)$ such that $Ap_i=q_i$ for $i=1,\ldots,4$.
\end{lemma}

\begin{proof}
  As no three of the points $p_i$ are collinear, and the same for $q_i$, there exists a unique $A\in\PGL_3(L)$ with $Ap_i=q_i$ for $i=1,\ldots,4$.
  Let $g\in\Gal(L/\kk)$.
  We will show that $g(A)=A$ (where $g(A)$ denotes the action of $g$ on the entries of $A$).
  By assumption, there exists $\sigma\in\Sym_4$ such that $g^{-1}(p_i)=p_{\sigma(i)}$ and $g^{-1}(q_i)=q_{\sigma(i)}$ for $i=1,\ldots,4$.
  We compute
  \[
    g(A)p_i = g(Ag^{-1}(p_i)) = g(Ap_{\sigma(i)}) = g(q_{\sigma(i)}) = q_i
  \]
  and so $g(A)p_i=q_i=Ap_i$ for $i=1,\ldots,4$.
  By the uniqueness of $A$, this gives $g(A)=A$.
\end{proof}

\begin{proposition}\label{proposition:SameFourPoints}
  Let $\kk$ be any perfect field.
  Let $\PPPP$ and $\QQQQ$ each be a union of four points in $\PPP^2$.
  Let $X_\PPPP\to\PPP^2$ and $X_\QQQQ\to\PPP^2$ be the blow-up centered at $\PPPP$ respectively $\QQQQ$ and let $\varphi\colon X_\PPPP\dasharrow X_\QQQQ$ be a birational map preserving the fibrations $X_\PPPP\to\PPP^1$ and $X_\QQQQ\to\PPP^1$ given by conics through $\PPPP$ respectively $\QQQQ$.
  Assume that one of the following holds: \begin{enumerate}
    \item\label{item:SameFourPoints--size4} $\PPPP$ is an orbit of size $4$, or
    \item\label{item:SameFourPoints--size22} $\PPPP$ consists of two orbits of size $2$.
  \end{enumerate}
  Then, there exists an automorphism $\alpha\in \Aut_\kk(\PPP^2)$ such that $\alpha(\PPPP)=\QQQQ$.
\end{proposition}

\begin{proof}
  Let $\PPPP=\{p_1,p_2,p_3,p_4\}$ and let $L/\kk$ be a Galois extension such that $\PPPP$ is invariant under its Galois action.
  Let $E_{ij}$ be the line through $p_i$ and $p_j$ for $j\neq i$ and assume that $p_1$ and $p_3$ are in the same orbit.
  With this assumption, Observations~\ref{observation:GaloisGroupOfOrbitSize4} and~\ref{observation:GaloisGroupOfOrbitSize2} yield that -- up to changing the numbering of the $p_i$ -- there exists a Galois action $\sigma_1$ mapping $E_{12}\leftrightarrow E_{34}$ and $E_{14}\leftrightarrow E_{23}$.
  Recall that on $X=X_\PPPP$, the $E_{ij}$ are irreducible components of three singular fibers, so they are vertical $(-1)$-curves.
  We denote the exceptional divisor of $p_i$ by $E_i$.
  Lemma~\ref{lemma:BlowUpOfFourPointsDescription} describes how the $E_k$ and $E_{ij}$ intersect:
  $E_k$ intersects $E_{ij}$ if and only if $k\in\{i,j\}$.\\
  Consider now the image by $\varphi$ of the vertical $(-1)$-curves. \begin{itemize}
    \item In case~\ref{item:SameFourPoints--size4}, all vertical $(-1)$-curves are mapped onto vertical $(-1)$-curves (since we can apply Lemma~\ref{lemma:IrreducibleComponentsExchanged} by Observation~\ref{observation:GaloisGroupOfOrbitSize4}).
    We write $F_{ij}=\varphi(E_{ij})$.
    \item In case~\ref{item:SameFourPoints--size22}, the image of $E_{13}$ or $E_{24}$ might be a point, but the image of the other vertical $(-1)$-curves are again vertical $(-1)$-curves (again we can apply Lemma~\ref{lemma:IrreducibleComponentsExchanged} by Observation~\ref{observation:GaloisGroupOfOrbitSize2}).
    So we set $F_{ij}=\varphi(E_{ij})$ for $(i,j)\notin \{(1,3),(2,4)\}$, and we denote by $F_{13}$ and $F_{24}$ the other two vertical $(-1)$-curves on $X_\QQQQ$ (which are defined over $\kk$).
  \end{itemize}

  So $F_{12}+F_{34}$, $F_{13}+F_{24}$, and $F_{14}+F_{23}$ are the three singular fibers on $X_\QQQQ$.
  Since $\sigma_1$ maps $E_{12}\leftrightarrow E_{34}$, it also maps $F_{12}\leftrightarrow F_{34}$ and Lemma~\ref{lemma:ThreeSingularFibers} implies that there are four $(-1)$-sections $F_1,\ldots, F_4$ that are the exceptional divisors corresponding to the points $q_1,\ldots, q_4$ of $\QQQQ$.
  As the set of the $F_{ij}$ is invariant under the Galois action $\Gal(L/\kk)$, also the set of the $F_i$ is invariant under the action, and so the set of points $q_1,\ldots,q_4$ is defined over $L$.
  As no three of the four points in $\PPPP$ respectively $\QQQQ$ are collinear (Lemma~\ref{lemma:BlowUpOfFourPointsDescription}), we prove that there exists an automorphism defined over $\kk$ that maps $\PPPP$ onto $\QQQQ$ using Lemma~\ref{lemma:AutoDefinedOverk}.
  It is enough to prove that for each $g\in\Gal(L/\kk)$ there exists $\sigma\in\Sym_4$ with $g(p_i)=p_{\sigma(i)}$ and $g(q_i)=q_{\sigma(i)}$.

  Let $\sigma\in\Gal(L/\kk)$.
  We consider $\sigma$ as an element in the symmetric group $\Sym_4$ via $\sigma(p_k)=p_{\sigma(k)}$.
  It remains to show that $\sigma(q_k)=q_{\sigma(k)}$ for $k=1,\ldots,4$.
  For this, we take a close look at the Galois action on the vertical $(-1)$ curves.
  Lemma~\ref{lemma:ThreeSingularFibers} tells us how the $F_i$ intersect the vertical $(-1)$-curves:
  \begin{enumerate*}[label=\rm(\roman*)] 
    \item
    $F_k$ intersects $F_{ij}$ if and only if $k\in\{i,j\}$, or
    \item
    $F_k$ intersects $F_{ij}$ if and only if $k\notin\{i,j\}$.
  \end{enumerate*}

  \begin{itemize}
    \item For case~\ref{item:SameFourPoints--size4}, remark that $\sigma(E_k)=E_{\sigma(k)}$ intersects $\sigma(E_{ij})=E_{\sigma(i)\sigma(j)}$ if and only if $\sigma(k)\in\{\sigma(i),\sigma(j)\}$.
    We distinguish cases i) and ii).
    \begin{enumerate}[label=\rm(\roman*)]
      \item $\sigma(k)\in\{\sigma(i),\sigma(j)\}$ holds if and only if $F_{\sigma(k)}$ intersects $F_{\sigma(i)\sigma(j)}$.
      The action of $\sigma$ onto $q_k$ is determined by the action of the three lines $F_{ij}$ with $k\in\{i,j\}$.
      Hence, $\sigma(q_k)=q_{\sigma(k)}$.
      \item $\sigma(k)\in\{\sigma(i),\sigma(j)\}$ holds if and only if $F_{\sigma(k)}$ does \textit{not} intersect $F_{\sigma(i),\sigma(j)}$.
      The action of $\sigma$ onto $q_k$ is determined by the action of the three lines $F_{ij}$ that do not contain $q_k$ (namely $k\notin\{i,j\}$) and so we have again that $\sigma(q_k)=q_{\sigma(k)}$.
    \end{enumerate}
    \item In case~\ref{item:SameFourPoints--size22}, we have that $E_{13}$ and $E_{24}$ are defined over $\kk$, so the union of the other four vertical $(-1)$-curves is invariant under the Galois action.
    As we have given the names $E_{13}$ and $E_{24}$ to the two vertical $(-1)$-curves that are defined over $\kk$ arbitrarily, cases i) and ii) are the same, up to reordering the $F_k$ (exchange $F_1$ with $F_3$, and $F_2$ with $F_4$).
    So we assume that the intersection properties of i) hold.
    So it is enough to consider the action on $E_{12}$, $E_{34}$, $E_{14}$, $E_{23}$.

    As $q_k$ is the intersection point of the two lines $F_{ki}$ and $F_{kj}$ from the above four lines, the action on $q_k$ is determined by the action on these two lines.
    So $\sigma(q_k)$ is the point of intersection of the lines $F_{\sigma(k)\sigma(i)}$ and $F_{\sigma(k)\sigma(j)}$, which is exactly $q_{\sigma(k)}$.
    Hence, $\sigma(q_k)=q_{\sigma(k)}$.%
  \end{itemize}
\end{proof}

\begin{lemma}\label{lemma:AllMoriConicBundles}
  Let $X\to\PPP^1$ be a rational Mori conic bundle.
  Then $X$ is either a Hirzebruch surface, or
  \begin{enumerate}
    \item\label{item:descriptionC5} $K_X^2=5$ and then there exists a link $X\to\PPP^2$ of type III, given by the blow-up of an orbit of size $4$, or
    \item\label{item:descriptionC6} $K_X^2=6$ and then $X$ is given by the blow-up of an orbit of two points in $\PPP^2$ followed by the contraction of the strict transform of the line through the two points \textup{(}this is a link of type II\textup{)}, followed by the blow-up of an orbit of size $2$.
  \end{enumerate}
  In particular, the fibration in cases~\ref{item:descriptionC5} and~\ref{item:descriptionC6} corresponds to the linear system of conics through the four blown up points.
\end{lemma}

\begin{proof}
  By Remark~\ref{remark:geography} (where we used Iskovskikh's list \cite[Theorem 2.6]{Iskovskikh96}), if $X$ is not a Hirzebruch surface then $K_X^2\in\{5,6\}$.

  First, assume that $K_X^2=5$.
  Remarks~\ref{remark:MFSToPointOrCurve} and~\ref{remark:ConicBundleOverP1} tell us that $X$ has three singular fibers and that there exists $\sigma\in\Gal(\bar\kk/\kk)$ that exchanges two components of a singular fiber.
  Hence, the assumptions of Lemma~\ref{lemma:ThreeSingularFibers} are satisfied and we find that $X\to\PPP^2$ is the blow-up centered at four points $q_1,\ldots,q_4$ whose union is defined over $\kk$.
  Iskovskikh's list also says that $X$ is obtained by the blow-up $X'\to\PPP^2$ of an orbit $\{p_1,\ldots,p_4\}$ of size $4$ in $\PPP^2$, followed by a sequence $X'\dashto X$ of links of type II between Mori conic bundles.
  So we can apply Proposition~\ref{proposition:SameFourPoints}, which implies in particular that $\{q_1,\ldots,q_4\}$ is again an orbit of size $4$ in $\PPP^2$.

  Now, we assume that $K_X^2=6$.
  As $X$ is rational, Iskovskikh's list gives that $X$ is obtained by the blow-up of an orbit of size $2$ in $\PPP^2$ followed by the contraction of the strict transform of the line between the two points (this is a link $Q\dashto\PPP^2$ of type II), followed by the blow-up $X'\to Q$ of an orbit of size $2$ (link of type III), followed by a sequence $X'\dashto X$ of links of type II between Mori conic bundles.
  Instead of this construction, we can consider $X$ to be obtained by the blow-up $Y'\to\PPP^2$ of two orbits $\{p_1,p_2\}$ and $\{p_3,p_4\}$ of size $2$ in $\PPP^2$, followed by a birational map $\varphi\colon Y'\dashto Y$ preserving the fibration, followed by the contraction $Y\to X$ of a singular fiber defined over $\kk$.
  Note that $Y$ is a conic bundle (but not a Mori conic bundle) with $K_Y^2=5$.
  With Observation~\ref{observation:GaloisGroupOfOrbitSize2} there is $\sigma_1\in\Gal(\bar\kk/\kk)$ that exchanges the irreducible components of one of the three singular fibers of $Y'$ and hence also of $Y$ (since singular fibers are mapped onto singular fibers by Corollary~\ref{corollary:NotAnIsoOfConicBundles}).
  So the assumptions of Lemma~\ref{lemma:ThreeSingularFibers} are satisfied and so $Y$ is the blow-up of four points $q_1,\ldots,q_4\in\PPP^2(\bar\kk)$.
  As the birational map $\varphi\colon Y'\dashto Y$ preserves the fibration, Proposition~\ref{proposition:SameFourPoints} gives an automorphism of $\PPP^2(\kk)$ defined over $\kk$ that maps $\{p_1,\ldots,p_4\}$ onto $\{q_1,\ldots,q_4\}$.
  Hence, also $\{q_1,\ldots,q_4\}$ consists of two orbits of size $2$.
\end{proof}

\begin{corollary}\label{corollary:GeneratingSets}
  Let $\kk$ be any perfect field.
  The group $\Bir_\kk(\PPP^2)$ is generated by the subsets $J_\ast$, $J_\circ$ and $G_{\leq 8}$ of $\Bir_\kk(\PPP^2)$, where \begin{enumerate}
    \item $J_\ast$ are the Jonquieres maps (i.e. the birational maps preserving a pencil of lines),
    \item $J_\circ=\bigcup J_\PPPP$ is the union of groups $J_\PPPP$ of birational maps that fix a pencil of conics through four points $\PPPP$, of which no three are collinear, and such that $\PPPP$ is either an orbit of size $4$ or consists of two orbits of size $2$,
    \item $G_{\leq 8}$ is the set of birational maps with Galois depth $\leq 8$.
  \end{enumerate}
\end{corollary}

\begin{proof}
  Let $f\in\Bir_\kk(\PPP^2)$.
  By \cite[Theorem 2.5]{Iskovskikh96}, the birational map $f$ can be written as a product of Sarkisov links of type I,II, III and IV.
  By Iskovskikh's classification of Sarkisov links in~\cite[Theorem 2.6]{Iskovskikh96}, the only link of type IV between rational surfaces is exchanging the fibration $\PPP^1\times\PPP^1\simto\PPP^1\times\PPP^1$, $[x_1:x_2;y_1:y_2]$.
  This can be decomposed as \[\PPP^1\times\PPP^1=\FFF_0 \stackrel{\psi}\dashto \FFF_1 \stackrel{\pi}{\to}  \PPP^2 \stackrel{\sigma}\simto \PPP^2 \stackrel{\pi^{-1}}{\dashto} \FFF_1 \stackrel{\psi^{-1}}\dashto \FFF_0 = \PPP^1\times\PPP^1,\]
  where $\psi$ is a link of type II, $\pi$ is the blow-up at one point (that is, of type III), and $\sigma$ is a change of coordinates of $\PPP^2$.
  Therefore, we can write $\varphi$ as a composition of Sarkisov links of type I, II, and III.

  As in Remark~\ref{remark:DecompositionOfBirMaps}, we can write \[f=F_{N+1}\circ\Phi_N\circ F_{N}\circ\cdots\circ F_2\circ\Phi_1\circ F_1,\]
  where $F_i\colon Y_{i-1}\dashrightarrow X_i$ are a composition of links of type I and III, as well as links of type II between Mori fiber spaces with $0$-dimensional base (In particular, the set of base points consists of unions of orbits of size $\leq 8$.), and $\Phi_i\colon X_i\dashrightarrow Y_i$ are a composition of links of type II between Mori conic bundles $X_i\to\PPP^1,Y_i\to\PPP^1$ for $i=1,\ldots,N$, and $Y_0=\PPP^2=X_{N+1}$.

  Lemma~\ref{lemma:AllMoriConicBundles} gives us a ``shortcut'' (that is, a birational map of Galois depth $\leq 8$) $\pi_{Y_i}\colon Y_i\dashto\PPP^2$ (respectively $\pi_{X_i}\colon X_i\dashto \PPP^2$) from any of these $X_i$, $Y_i$ to $\PPP^2$.
  We let $f_i=\pi_{X_i}\circ F_i\circ(\pi_{Y_{i-1}})^{-1}\in \Bir(\PPP^2)$ and note that it has Galois depth $\leq 8$, where the $\pi_{X_i}$ and $\pi_{Y_i}$ are the ``shortcuts'' to $\PPP^2$ from above.
  Hence, $f_i\in G_{\leq 8}$.
  We can thusly write \[f=f_1\circ \pi_{X_1}^{-1}\circ\Phi_1\circ\pi_{Y_1}\circ f_2\circ\pi_{X_2}^{-1}\circ\Phi_2\circ\cdots\circ\Phi_N\circ\pi_{Y_N}\circ f_{N+1}.\]
  Therefore, it is enough to prove that any element of the form \[f=\pi_{X}^{-1}\circ\Phi\circ\pi_Y\] lies in $\langle J_\ast,J_\circ,G_{\leq8} \rangle$ (and then proceed by induction).
  Recall that $\Phi=\varphi_n\circ\cdots\circ\varphi_1$ is a composition of links $\varphi_i\colon Z_{i-1}\dashrightarrow Z_{i}$ of type II between Mori conic bundles $Z_i\to\PPP^1$ with $K_{Z_i}^2\in\{5,6,8\}$ (see Lemma~\ref{lemma:AllMoriConicBundles}), and $Z_0=X$, $Z_n=Y$.
  As each of the $\varphi_i$ preserves the fibration, so does $\Phi$.

  If $X$ and $Y$ are Hirzebruch surfaces, then $X=Y=\FFF_1$ and the $Z_i$ are some Hirzebruch surfaces. As $\pi_X\colon X\to\PPP^2$ is the blow-up of one point $p$ in $\PPP^2$, the fibers of $X$ correspond to the pencil of lines through $p$.
  Similarly, the fibers of $Y$ correspond to the pencil of lines through some point $q$ (and by applying an element of $\Aut(\PPP^2)$ we can assume $q=p$).
  As $\Phi$ preserves the fibration, the pencil of lines through $p$ is mapped under $f$ onto the pencil of lines through $p$, so $f$ is the product of an element of $\Aut(\PPP^2)$ and one of $J_\ast$.

  If $X$ has $K_X^2=5$, we let $\PPPP\subset\PPP^2$ be the set of base points of $\pi_X^{-1}$, which is an orbit of size $4$.
  As in Remark~\ref{remark:BlowUpOf4Pts}, the fibration of $X\to\PPP^1$ corresponds to the pencil of conics through $\PPPP$.
  Analogously, let $\QQQQ\subset\PPP^2$  be the set of base points of $\pi_Y^{-1}$, which is again an orbit of size $4$.
  The fibration $Y\to\PPP^1$ corresponds to the pencil of conics through $\QQQQ$.
  As $\Phi$ preserves the fibration and no base point of $\Phi$ lies on a singular fiber by \cite[Theorem 2.6]{Iskovskikh96}, the singular fibers are preserved.
  By Proposition~\ref{proposition:SameFourPoints}, there exists an automorphism $\alpha\in\Aut_{\kk}(\PPP^2)$ such that $\alpha(\PPPP)=\QQQQ$.
  Therefore, $f$ is the product of an element of $\Aut(\PPP^2)$ and one of $J_\PPPP$.

  If $X$ has $K_X^2=6$, we let $\PPPP\subset\PPP^2$ be the set of base points of $\pi_X^{-1}$, which consists of two orbits of size $2$.
  We can argue as before and find that $f$ is the product of an element of $\Aut(\PPP^2)$ and one in $J_\PPPP$.
\end{proof}

\subsection{Free product}\label{section:freeproduct}

\begin{lemma}\label{lemma:IrreducibleConicRespLine}
  Let $\PPPP=\{p_1,\ldots,p_4\}$ be a set of four points in $\PPP^2(\bar\kk)$, no three collinear.
  \begin{enumerate}
    \item\label{item:IrreducibleConicRespLine--4} If $\PPPP$ is an orbit of size $4$ of the action of $\Gal(\bar\kk/\kk)$, then there exists $A\in\PGL_3(\kk)$ such that $Ap_i=[1:a_i:a_i^2]$ with $a_i\in\bar\kk$ forming an orbit $\{a_1,\ldots,a_4\}\subset\bar\kk$ of size $4$ under the Galois action.
    \item\label{item:IrreducibleConicRespLine--22}
    If $\{p_1,p_2\}$ and $\{p_3,p_4\}$ form two orbits of size $2$, then there exists $A\in\PGL_3(\kk)$ such that $Ap_i=[1:a_i:0]$ for $i=1,2$ and $Ap_i=[1:0:a_i]$ for $i=3,4$ with $a_i\in\bar\kk$ and $\{a_1,a_2\}$ as well as $\{a_3,a_4\}$ form an orbit in $\bar\kk$.
  \end{enumerate}
  In both cases, the field of definition of $\PPPP$ is $\kk(a_1,\ldots,a_4)$.
\end{lemma}

\begin{proof}
  For \ref{item:IrreducibleConicRespLine--22} it is enough to observe that the line $L$ through $p_1$ and $p_2$ is invariant under the Galois action, hence it is defined over $\kk$. The same holds for the line $L'$ through $p_3$ and $p_4$. Therefore, there exists an automorphism defined over $\kk$ that maps $L$ onto $z=0$ and $L'$ onto $y=0$.

  For \ref{item:IrreducibleConicRespLine--4}, we remark that if there exists a $\kk$-point $p$ that is not collinear with any two of the four points, then the unique conic $C$ through the five points is irreducible.
  Since the set of the five points is invariant under the action of $\Gal(\bar\kk/\kk)$, the conic is defined over $\kk$.
  Moreover, the conic contains a $\kk$-point and so there exists an automorphism of $\PPP^2$ defined over $\kk$ that maps $C$ onto the conic $xz-y^2=0$, giving the desired coordinates of the $p_i$.
  So it remains to prove that there exists a $\kk$-point $p$ that does not lie on any of the six lines through two of the points of $\PPPP$.
  By contradiction, assume that each point in $\PPP^2(\kk)$ is collinear with two of the points in $\PPPP$.
  Hence, all points of $\PPP^2(\kk)$ lie on the union of the six lines $L_{ij}$ through two points $p_i$ and $p_j$.
  Note that for any field $\kk$, $\PPP^2(\kk)$ consists of at least $7$ points.
  Hence, there are two points in $\PPP^2(\kk)$ that lie on the same line $L_{ij}$, say $L_{12}$.
  So $L_{12}$ is invariant under the Galois action, and hence the set $\{p_1,p_2\}$ is defined over $\kk$.
  This is a contradiction to $\PPPP$ being a Galois orbit of size $4$.
\end{proof}

\begin{example}[Construction of a link with Galois depth $2n+1$ for the case $K_X^2=5$]\label{example:ConstructionLink4}
  Let $\kk$ be a perfect field with an orbit $\PPPP=\{p_1,\ldots,p_4\}$ in $\PPP^2$ of size $4$, no three collinear.
  We assume that there is an irreducible polynomial in one variable of odd degree $2n+1$ with roots $r_1,\ldots,r_{2n+1}\in\bar\kk$.
  With Lemma~\ref{lemma:IrreducibleConicRespLine} we can assume that $p_i=[1:t_i:t_i^2]$.
  Consider the $2n+1$ points $q_i=[0:1:r_i]$ (all lying on the line $L$ defined by $x=0$) and the conics $C_i$ given uniquely by the $5$ points $p_1,p_2,p_3,p_4$ and $q_i$.
  We want to show that all the conics $C_i$ are distinct.

  Assume by contradiction that there exist $i\neq j$ with $q_j\in C_i$, so $C_i\cap L=\{q_i,q_j\}$.
  As the Galois group acts transitively on the $q_k$, and $L$ is invariant under the Galois action, for all $k$ we have $C_k\cap L=\{p_k,p_l\}$ for a $l\neq k$.
  Since the orbit $\QQQQ=\{q_1,\ldots,q_{2n+1}\}$ consists of an odd number of points, this is a contradiction.

  Let $X\to\PPP^2$ be the blow-up at $\PPPP$, so $X\to\PPP^1$ is a Mori conic bundle with $K_X^2=5$.
  As the conics $C_i$ are distinct, the orbit $\QQQQ$ of size $2n+1$ from above gives an orbit in $X$ such that each point of the orbit lies on a different fiber.
  Hence, the birational map $X\dashto X'$ that is given by the blow-up of $\QQQQ$ (seen as an orbit in $X$) followed by the contraction of the orbit containing the strict transform of the fibers through the points of $\QQQQ$ is a link of type II between Mori conic bundles, and it has Galois depth $2n+1$.
  \begin{center}
    \begin{tikzpicture}[scale=0.5]
      \draw[name path=ellipse1, rotate=-20] (0,0.5) ellipse (4 and 1);
      \draw[name path=ellipse2, rotate=50] (0,-0.5) ellipse (4 and 1);
      \draw[name path=line] (-3,3) -- (-1,-5);
      \path [name intersections={of=ellipse1 and ellipse2}];
      \node [fill=black,inner sep=1pt,label=-90:$p_1$] at (intersection-1) {};
      \node [fill=black,inner sep=1pt,label=-90:$p_2$] at (intersection-2) {};
      \node [fill=black,inner sep=1pt,label=-90:$p_3$] at (intersection-3) {};
      \node [fill=black,inner sep=1pt,label=-90:$p_4$] at (intersection-4) {};
      \path [name intersections={of=line and ellipse1}];
      \node [fill=black,inner sep=1pt,label=-180:$q_1$] at (intersection-1) {};
      \node [fill=black,inner sep=1pt,label=-180:$q_2$] at (intersection-2) {};
      \path [name intersections={of=line and ellipse2}];
      \node [fill=black,inner sep=1pt,label=-180:$q_{2n-1}$] at (intersection-1) {};
      \node [fill=black,inner sep=1pt,label=-180:$q_{2n}$] at (intersection-2) {};
      \draw [name path=fakeline1, draw=none] (0,0) -- (-5,0);
      \draw [name path=fakeline2, draw=none] (0,0) -- (-5,-2);
      \draw [name path=fakeline3, draw=none] (0,0) -- (-1.2,-5);
      \path [name intersections={of=line and fakeline1}];
      \node [fill=black,inner sep=1pt] at (intersection-1) {};
      \path [name intersections={of=line and fakeline2}];
      \node [fill=black,inner sep=1pt] at (intersection-1) {};
      \path [name intersections={of=line and fakeline3}];
      \node [fill=black,inner sep=1pt,label=-180:$q_{2n+1}$] at (intersection-1) {};
    \end{tikzpicture}
  \end{center}

\end{example}

\begin{example}[Construction of a link with Galois depth $2n+1$ for the case $K_X^2=6$]\label{example:ConstructionLink22}
  Let $\kk$ be a perfect field with two orbits $\{p_1,p_2\}$ and $\{\tilde p_1,\tilde p_2\}$ in $\PPP^2(\bar\kk)$ of size $2$.
  We assume that there is an irreducible polynomial in one variable of odd degree $2n+1$ with roots $r_1,\ldots,r_{2n+1}\in\bar\kk$.
  With Lemma~\ref{lemma:IrreducibleConicRespLine} we can assume that $p_i=[1:a_i:0]$ for an orbit $\{a_1,a_2\}\subset\bar\kk$ and $\tilde p_i=[1:0:b_i]$ for an orbit $\{b_1,b_2\}\subset\bar\kk$.
  So $\PPPP=\{p_1,p_2,\tilde p_1,\tilde p_2\}$ consists of two orbits of size two, and no three of the four points are collinear.
  The points $q_i=[0:1:r_i]$ form an orbit of size $2n+1$.
  As none of the $q_i$ is collinear with two of the four points in $\PPPP$, we consider the irreducible conic $C_i$ through the points $q_i$ and $\PPPP$, which is unique.
  These conics are all distinct by the same argument as in the example above.

  Let $X\dashto\PPP^2$ be the blow-up at $\PPPP$, followed by the contraction of the line through the two points of one of the orbits of size $2$. Then, $X\to\PPP^1$ is a Mori conic bundle with $K_X^2=6$.
  So the birational map $X\dashto X'$ that is given by the blow-up of the orbit $\{q_1,\ldots,q_{2n+1}\}$, followed by the contraction of the strict transform of the conics $C_1,\ldots,C_{2n+1}$ is a link of type II between Mori conic bundles of Galois depth $2n+1$.
\end{example}

\begin{proof}[Proof of Theorem~\textup{\ref{theorem:Refinement}}]
  We consider the group homomorphism from $\Bir(\PPP^2)$ to a free product (indexed by equivalence classes of Mori conic bundles) of direct sums (indexed by equivalence classes of Sarkisov links of type II) of $\ZZZ/2\ZZZ$ constructed in Theorem~\ref{theorem:GroupHomomorphism}.
  In Theorem~\ref{theorem:CremonaNotSimple} we have already seen that the restriction to the equivalence class of Mori conic bundles corresponding to the Hirzebruch surfaces gives an infinite direct sum of $\ZZZ/2\ZZZ$, providing the infinite set $I_0$.
  In the following we construct $\NNNN_4$ distinct equivalence classes of Mori conic bundles $X$ with $K_X^2=5$, and $\NNNN_2$ with $K_X^2=6$, and remark with Examples~\ref{example:ConstructionLink4} and~\ref{example:ConstructionLink22} that there are (at least) as many distinct equivalence classes of links of type II as there are irreducible polynomials of odd degree in $\kk[x]$.
  Projecting onto the corresponding factors gives the desired group homomorphism.
  \begin{itemize}
    \item
    To a Galois orbit $\PPPP$ of size $4$ in $\PPP^2$ we associate the equivalence class of its Mori conic bundle $X\to\PPP^1$ with $K_X^2=5$.
    Note that if there is another orbit $\QQQQ$ of size $4$ in $\PPP^2$ such that no element of $\PGL_3(\kk)$ maps $\PPPP$ onto $\QQQQ$, then the obtained Mori conic bundles are not equivalent by Proposition~\ref{proposition:SameFourPoints}.
    This gives us $\mathcal{N}_4$ distinct equivalence classes of Mori conic bundles, explaining the cardinality of $J_5$.

    For each $n\geq 8$ such that there exists an irreducible polynomial of degree $2n+1$, Example~\ref{example:ConstructionLink4} gives a link $X\dashto X'$ of type II of Galois depth $2n+1\geq 17$ between two Mori conic bundles in the equivalence class.
    \item
    Let $\{p_1,p_2\}\subset\PPP^2$ be an orbit of size $2$.
    As the line through the two points $p_1$ and $p_2$ is defined over $\kk$, we can assume it to be the line $z=0$, and so $p_i=[1:a_i:0]$ for an orbit $\{a_1,a_2\}\subset\bar\kk$.
    Let $\tilde p_i=[1:0:a_i]$ for $i=1,2$.
    So $\PPPP=\{p_1,p_2,\tilde p_1,\tilde p_2\}$ consists of two orbits of size two, and no three of the four points are collinear.
    Consider the Mori conic bundle $X\to\PPP^1$ with $K_X^2=6$ associated to the two orbits of size $2$.
    Note that if we perform the same construction from another orbit $\{q_1,q_2\}$ of size $2$ in $\PPP^2$ such that there is no element in $\PGL_3(\kk)$ that maps $\{p_1,p_2\}$ onto $\{q_1,q_2\}$, then the obtained Mori conic bundles are not equivalent by Proposition~\ref{proposition:SameFourPoints}.
    This gives us $\mathcal{N}_2$ distinct equivalence classes of Mori conic bundles, explaining the cardinality of $J_6$.

    Example~\ref{example:ConstructionLink22} gives a link $X\dashto X'$ of type II of Galois depth $2n+1\geq 17$ between two Mori conic bundles in the equivalence class of $X$.
  \end{itemize}
  With Lemma~\ref{lemma:AllMoriConicBundles} these links extend to a birational map $\PPP^2\dashto\PPP^2$ that factorizes through Sarkisov links with Galois depth $\leq 4$ and one link of Galois depth $2n+1\geq 17$, corresponding to the generator indexed by its equivalence class.
  This implies that $|I|$ is the number of integers $n\geq8$ such that there exists an irreducible polynomial of degree $2n+1$.

  Finally, note that for all finite fields there is a (unique) field extension of degree $2$ and $4$.
  For an irreducible polynomial of degree $4$ with roots $a_1,\ldots,a_4$ we get an orbit $[1:a_i:a_i^2]$ of size $4$ in $\PPP^2$, and for an irreducible polynomial of degree two with roots $a_1,a_2$ we take the orbit $[1:a_i:0]$ of size $2$.
  So we have that $\mathcal{N}_2$ and $\mathcal{N}_4$ are at least $1$.
  (In fact, both are equal to one with~\cite[Lemma 4.1]{ALNZ19}.)
  Over the rational numbers $\QQQ$ there are infinitely many distinct field extensions of degree $2$ and of degree $4$.
  We find that $\NNNN_4$ and $\NNNN_2$ are infinite.
\end{proof}

\newpage
\subsection{Long list of Sarkisov links in a nutshell}\label{section:longlist}

All Sarkisov links between rational surfaces listed in \cite[Theorem 2.6]{Iskovskikh96} are given by links in Figure~\ref{figure:longlist} with notation of Table~\ref{table:longlist}.

\begin{table}[H]
  \caption{Notation to Figure~\ref{figure:longlist}}
  \label{table:longlist}
  \centering
  \begin{tabular}{cl}
    $\mathcal{D}_d$ & Set of del Pezzo surfaces of degree $d$.\\ \addlinespace[0.1cm]
    $\mathcal{C}_d$ & Set of Mori conic bundles $X\to\PPP^1$ with $K_X^2=d$.\\ \addlinespace[0.1cm]
    $X\stackrel{\text{I:a}}{\dashto}Y$& \makecell[l]{Link of type I from $X$ to $Y$ that is given by the blow-up\\ $Y\to X$ of an orbit in $X$ of size $a$.}\\ \addlinespace[0.1cm]
    \makecell{$X\stackrel{\text{II:a:b}}\dashrightarrow Y$} & \makecell[l]{Link of type II from $X$ to $Y$ that is given by the blow-up\\ of an orbit in $X$ of size $a$, followed by the contraction of an\\ orbit of size $b$ to $Y$.\\}
  \end{tabular}
\end{table}

\begin{figure}[h]
  \begin{center}
    \includegraphics[scale=0.75]{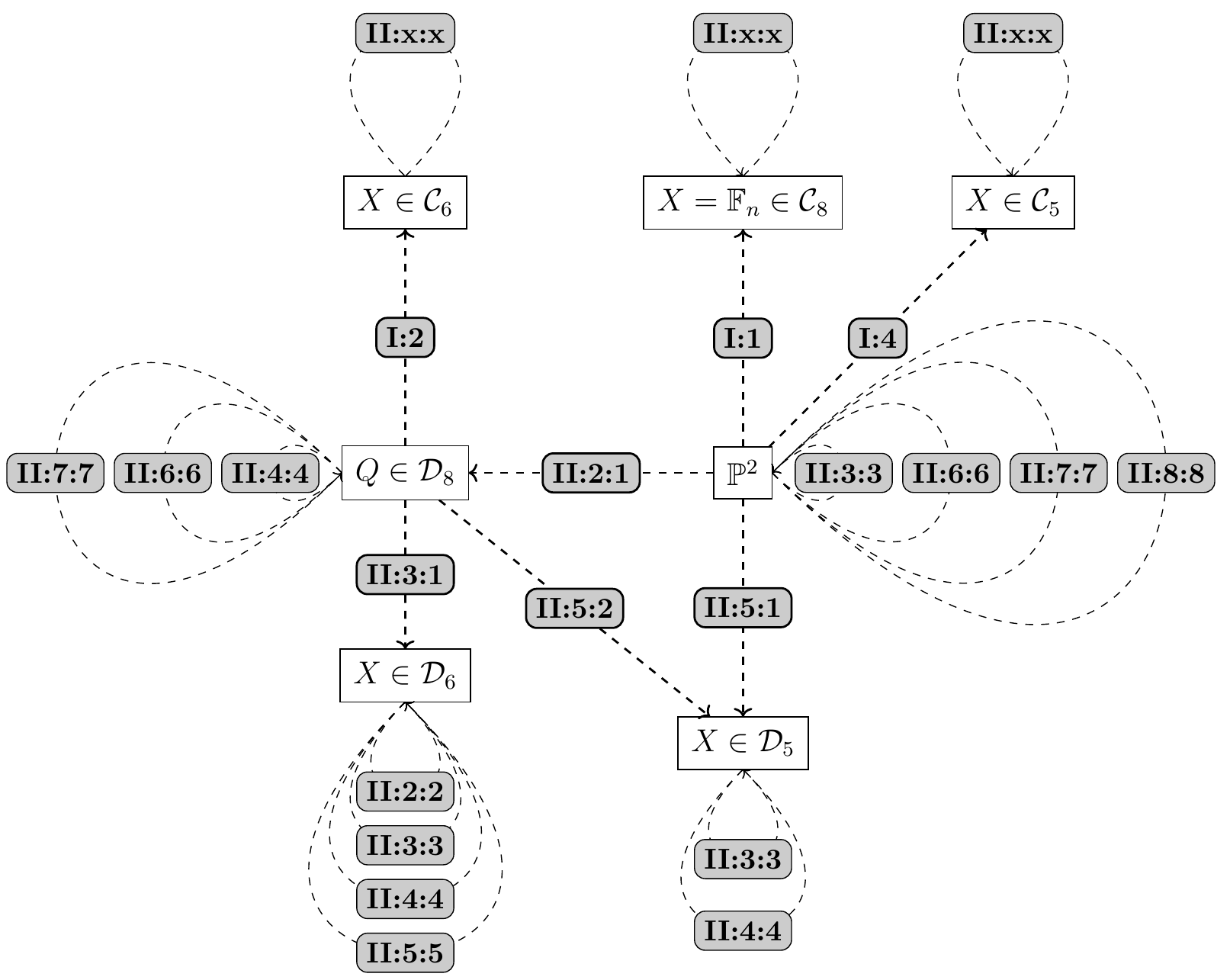}%
  \end{center}
  \caption[]{The long list of Sarkisov links, using the notation of Table~\ref{table:longlist}.
  The links of type II between the conic bundles are special as they allow a link of type II:$x$:$x$ for any $x\geq1$.}

  \label{figure:longlist}
\end{figure}



\begin{thebibliography}{ALNZ19}

\bibitem[ALNZ19]{ALNZ19}
Shamil Asgarli, Kuan-Wen Lai, Masahiro Nakahara, and Susanna Zimmermann.
\newblock Biregular {Cremona} transformations of the plane, 2019.

\bibitem[Bla09]{Blanc09}
J{\'e}r{\'e}my Blanc.
\newblock Linearisation of finite abelian subgroups of the {Cremona group} of
  the plane.
\newblock {\em Groups, geometry and dynamics}, 3:215--266, 05 2009.

\bibitem[BLZ19]{BLZ19}
J{\'e}r{\'e}my {Blanc}, St{\'e}phane {Lamy}, and Susanna {Zimmermann}.
\newblock {Quotients of higher dimensional Cremona groups}.
\newblock {\em arXiv e-prints}, page arXiv:1901.04145, January 2019.

\bibitem[Cas01]{Castelnuovo1901}
Guido Castelnuovo.
\newblock {\em Le trasformazioni generatrici del gruppo cremoniano nel piano}.
\newblock Turin R. Accad. d. Sci., 1901.

\bibitem[CL13]{CL13}
Serge Cantat and St\'{e}phane Lamy.
\newblock Normal subgroups in the {C}remona group.
\newblock {\em Acta Math.}, 210(1):31--94, 2013.
\newblock With an appendix by Yves de Cornulier.

\bibitem[{Cor}95]{Corti95}
Alessio {Corti}.
\newblock {Factoring birational maps of threefolds after Sarkisov. Appendix:
  Surfaces over nonclosed fields.}
\newblock {\em {J. Algebr. Geom.}}, 4(2):223--254, appendix 248--254, 1995.

\bibitem[Dol12]{Dolgachev12}
Igor~V. Dolgachev.
\newblock {\em Classical algebraic geometry}.
\newblock Cambridge University Press, Cambridge, 2012.
\newblock A modern view.

\bibitem[Hig64]{higgins64}
P.~J. Higgins.
\newblock Presentations of groupoids, with applications to groups.
\newblock {\em Mathematical Proceedings of the Cambridge Philosophical
  Society}, 60(1):07–20, 1964.

\bibitem[IKT93]{IKT93}
Vasily~A. {Iskovskikh}, Farkhat~K. Kabdykairov, and Semion~L. Tregub.
\newblock Relations in a two-dimensional {C}remona group over a perfect field.
\newblock {\em Izv. Ross. Akad. Nauk Ser. Mat.}, 57(3):3--69, 1993.

\bibitem[{Isk}96]{Iskovskikh96}
Vasily~A. {Iskovskikh}.
\newblock {Factorization of birational maps of rational surfaces from the
  viewpoint of Mori theory}.
\newblock {\em Russian Mathematical Surveys}, 51:585--652, August 1996.

\bibitem[Kol99]{kollar99}
J{\'a}nos Koll{\'a}r.
\newblock {\em Rational Curves on Algebraic Varieties}.
\newblock Ergebnisse der Mathematik und ihrer Grenzgebiete. 3. Folge / A Series
  of Modern Surveys in Mathematics. Springer Berlin Heidelberg, 1999.

\bibitem[Lan05]{lang05}
Serge Lang.
\newblock {\em Algebra}.
\newblock Graduate Texts in Mathematics. Springer New York, 2005.

\bibitem[Lon16]{lonjou16}
Anne Lonjou.
\newblock {Non simplicit{\'e} du groupe de Cremona sur tout corps}.
\newblock {\em Annales de l'Institut Fourier}, 66(5):2021--2046, 2016.

\bibitem[LZ19]{LZ19}
St{\'e}phane {Lamy} and Susanna {Zimmermann}.
\newblock {Signature morphisms from the Cremona group over a non-closed field}.
\newblock {\em Journal of the European Mathematical Society}, to appear 2019.

\bibitem[Noe70]{Noether1870}
Max Noether.
\newblock Ueber {F}l\"{a}chen, welche {S}chaaren rationaler {C}urven besitzen.
\newblock {\em Math. Ann.}, 3(2):161--227, 1870.

\bibitem[Sch73]{schupp73}
Paul~E. Schupp.
\newblock A survey of {SQ}-universality.
\newblock In {\em Conference on {G}roup {T}heory ({U}niv.
  {W}isconsin-{P}arkside, {K}enosha, {W}is., 1972)}, pages 183--188. Lecture
  Notes in Math., Vol. 319. 1973.

\bibitem[{She}13]{SB13}
Nicholas~I. {Shepherd-Barron}.
\newblock {Some effectivity questions for plane Cremona transformations}.
\newblock {\em arXiv e-prints}, page arXiv:1311.6608, Nov 2013.

\bibitem[Zim18]{zimmermann18}
Susanna Zimmermann.
\newblock The {A}belianization of the real {C}remona group.
\newblock {\em Duke Math. J.}, 167(2):211--267, 2018.

\end{thebibliography}

\end{document}